\documentclass[12pt]{article}

\usepackage{amsmath, amssymb,epsfig,bm}
\usepackage{latexsym}
\usepackage{algorithmicx}
\usepackage[ruled,vlined,linesnumbered]{algorithm2e}

\usepackage{graphicx,subfigure}
\usepackage{color}
\usepackage{appendix}

\numberwithin{equation}{section}
\usepackage[margin=1.0in]{geometry}    

\newtheorem{lemma}{Lemma}
\newtheorem{prop}[lemma]{Proposition}

\newtheorem{theorem}[lemma]{Theorem}
\newtheorem{rem}[lemma]{Remark}

\newtheorem{assu}[lemma]{Assumption}

\newcommand{\Om}{\Omega}

\newcommand{\lam}{\lambda}

\newcommand{\be}{\begin{eqnarray}}
\newcommand{\ee}{\end{eqnarray}}
\newcommand{\beq}{\begin{equation}}
\newcommand{\eeq}{\end{equation}}
\newcommand{\ben}{\begin{eqnarray*}}
\newcommand{\een}{\end{eqnarray*}}

\numberwithin{lemma}{section}

\begin{document}
\title{A data-driven model reduction method for parabolic inverse source problems and its convergence analysis}
\author{
Zhongjian Wang\thanks{ Department of Statistics and CCAM, The University of Chicago, Chicago, IL 60637, USA. (zhongjian@statistics.uchicago.edu).
}
\and
Wenlong Zhang\thanks{Corresponding author. Department of Mathematics, Southern University of Science and Technology (SUSTech), 1088 Xueyuan Boulevard, University Town of Shenzhen, Xili, Nanshan, Shenzhen, Guangdong Province, P.R.China. (zhangwl@sustech.edu.cn).
}
\and
Zhiwen Zhang\thanks{Corresponding author. Department of Mathematics, The University of Hong Kong, Pokfulam Road, Hong Kong SAR, P.R.China. (zhangzw@hku.hk).
}}

\date{}
\maketitle

\begin{abstract}
In this paper, we propose a data-driven model reduction method to solve parabolic inverse source problems efficiently. Our method consists of offline and online stages. In the off-line stage, we explore the low-dimensional structures in the solution space of the parabolic partial differential equations (PDEs) in the  forward problem with a given class of source functions and construct a small number of proper orthogonal decomposition (POD) basis functions to achieve significant dimension reduction. Equipped with the POD basis functions, we can solve the forward problem extremely fast in the online stage. Thus, we develop a fast algorithm to solve the optimization problem in the parabolic inverse source problems, which is referred to as the POD algorithm in this paper. Under a weak regularity assumption on the solution of the parabolic PDEs, we prove the convergence of the POD algorithm in solving the forward parabolic PDEs. In addition, we obtain the error estimate of the POD algorithm for parabolic inverse source problems. Finally, we present numerical examples to demonstrate the accuracy and efficiency of the proposed method.  Our numerical results show that the POD algorithm provides considerable computational savings over the finite element method.

\end{abstract}




\section{Introduction}

This paper presents a data-driven model reduction method to solve parabolic inverse source problems and studies the convergence analysis of the proposed method. Inverse problems are very important in physics, engineering, and bioengineering. The inverse source problems, which seeks reconstruction of source from final time observation, have attached much attention of the researchers over the past decades, see an introduction and references in \cite{Isakov2013}. They have been widely studied in the literature, applied to many physical and engineering source identification problems, e.g. migration of groundwater, groundwater pollution detection, control of pollution source and environmental protection  \cite{CHENG2020106213, new3, GER83, new2, new6} and references therein. The accurate recovery pollutant source is crucial to environmental safeguard in cities with high populations \cite{CHENG2010142}. The estimation of the strength of acoustic sources from measurements can be found in e.g. \cite{new6, Badia2011, nelson20}.  The inverse source problems that arise from PDEs are ill-posed in the sense of Isakov and Hadamard  \cite{new6,hadamard}, since the eigenvalues of the elliptic operators decay exponentially fast, especially the lack of stability with respect to the uncertainty in the measurement data is the most difficult challenge for numerical inversions. Namely, a small change in the data may lead to a significant difference in the reconstructed source strength.  Due to the important applications of inverse source problems, numerical methods have been widely explored \cite{CHENG2020106213,Chen-Zhang2021,LZ07,new4,GER83} and references can be found therein. 

 In this paper, we will work on a very practically physical scenario, assuming that the observational measurement is collected point-wisely over a set of distributed sensors located at $\{x_i\}^n_{i=1}$ over the physical domain  \cite{new2, new6, L08, NNR98}. At each sensor, independently additional noise or random error will be considered due to the uncertainty of natural noise, measurement errors, and the other uncertainty of the model itself. We will take one of the realistic approaches for such inverse problems, i.e. optimizing the mean-square error with a proper Tikhonov type regularizations \cite{Chen-Zhang2021, GER83, WY10}. Classical methods such as regression methods \cite{GER83}, linear and nonlinear programming methods \cite{GER83}, linear and nonlinear conjugate gradient methods \cite{AB05, WY10}, Newton-type methods, etc. can be applied during optimization.  For the inverse parabolic source problem, one usually uses iteration optimization methods to find the true source term \cite{Chen-Zhang2021, Johansson2014, Johansson2007}. In each iteration, one has to solve the forward parabolic equation one or two times. However, as the sizes of discrete problems grow (e.g. finite element method (FEM) or finite difference method (FDM)), the computation time will grow rapidly, especially for such evolution problems. As a consequence, the computation of the forward equation will cost the most time throughout the whole procedure. This motivates us to develop efficient model reduction methods to address this issue.

One of the successful model reduction ideas in solving evolution problems is the proper orthogonal decomposition (POD) method \cite{sirovich1987,berkooz1993POD}. The POD method uses the data from an accurate numerical simulation and extracts the most energetic modes in the system by using the singular value decomposition. This approach generates low-dimensional structures that can approximate the solutions of the evolution problems with high accuracy. The POD method has been used to solve many types of PDEs, including linear parabolic equations \cite{volkwein2013proper},  Navier‐Stokes equations \cite{kunisch2001galerkin}, viscous G-equations \cite{gu2021error},  Hamilton–Jacobi–Bellman (HJB) equations \cite{kunisch2004hjb}, and optimal control problem \cite{alla2013time}. The interested reader is referred to \cite{quarteroni2015reduced,benner2015survey,hesthaven2016certified} for a comprehensive introduction of the model reduction methods.

Since the POD method can significantly accelerate the computation of the forward problem compared with the traditional methods, e.g. FEM and FDM, we apply the POD method to solve parabolic inverse source problems with scattered measured data at the final time. 
The key idea is to exploit (model-based)  and construct  (data-based)  the intrinsic approximate low-dimensional structures of the underlying problem that consists of two components.  First, we have a  training component that computes a set of data-driven basis functions (i.e. the POD basis functions) to achieve significant dimension reduction in the solution space. Following up, there is a fast-solving component that solves the optimization problem using the constructed POD basis functions in each iteration.  Hence we achieve an effective data-driven model reduction method in solving the parabolic inverse source problems and overcome the typical computational bottleneck in the traditional methods, e.g. FEM and FDM. 

Then, we study the convergence analysis of the proposed method. In \cite{Chen-Zhang2021}, a general theory of the stochastic convergence of numerical method has been established for a certain type of inverse problem. Based on this framework, we explore the convergence theory in the POD settings.  Specifically, we will prove a relatively weaker convergence of the POD method when the source term only belongs to $L^2(\Omega)$, where $\Omega$ is the computational domain.  Unlike the traditional analysis of the POD method \cite{kunisch2001galerkin} or FEM convergence, we do not assume the higher regularity for parabolic PDE solution $u$, i.e. $u_{tt}$ to be bounded in $L^2(\Omega)$, which is quite strict in many cases. Based on our analysis, we derive the stochastic convergence when applying the POD method to the parabolic inverse source problem with uncertain data. Our analysis in Theorem \ref{thm_convergence} shows that the optimal error of the Tikhonov type least-square optimization problem depends on the noise level, the number of sensors, and without any higher source regularization but the $L^2(\Omega)$ norm of the source term. A self-consistent algorithm to get the optimal smoothing parameter is also established using the POD method to solve the inverse parabolic source problem, motivated by the recent study in \cite{Chen-Zhang, Chen-Zhang2021}. Finally, we conduct numerical experiments to demonstrate the accuracy and efficiency of the proposed method. Several kinds of source functions are involved. Comparing to original optimization with FEM basis functions, the proposed POD method significantly reduces the computational cost and maintains the same accuracy as the FEM.
  
The rest of the paper is organized as follows. In Section 2, we introduce the setting of the parabolic inverse source problems. In Section 3, we develop the data-driven model reduction method for solving parabolic inverse source problems. In Section 4, we provide the error estimate for the proposed method. In Section 5, we present numerical results to demonstrate the accuracy of our method. Finally, concluding remarks are made in Section 6.
\section{Parabolic inverse source problems}
To start with, we consider a parabolic equation, 
\begin{equation}\label{zz0}
	\left\{
	\begin{aligned}
		u_t +\mathcal{L}u &= f(x) &\mbox{in } \Omega\times (0, T), \\
		u(x, t)&= 0  &\mbox{on } \partial \Omega\times (0, T),\\
		u(x, 0)&= g(x) &\mbox{in } \Omega\,,
	\end{aligned} 
	\right.
\end{equation}
where $\Omega\subset \mathbb R^d$ $(d=1,2,3)$ is a bounded domain with $C^2$ boundary or a convex domain satisfying the uniform cone condition, $\mathcal{L}$ is a second-order elliptic operator given by $\mathcal{L}u=-\nabla\cdot (a(x)\nabla u) +c(x)u$, and
$g(x)$ is the initial value. We assume the  elliptic operator $\mathcal{L}$ is uniform elliptic, i.e. there exist $a_{\min}, a_{\max}>0$, such that $a_{\min}<a(x)<a_{\max}$ for all $x \in \Omega$. Moreover, we assume $a(x)\in C^{1}(\bar{\Omega})$, $c(x)\in C(\bar\Omega)$ and $c(x)\geq 0$. 

Let $u$ be the solution of the parabolic equation \eqref{zz0}. We define the forward operator $\mathcal{S}:$ $L^2(\Omega)\rightarrow H^2(\Omega)$ by $\mathcal{S}f=u(\cdot,T)$. The forward problem is to compute the solution $u(\cdot,t)$ for $t>0$ with known source term $f(x)$ and known initial condition $g(x)$. 


In the parabolic inverse source problem, $f(x)\in L^2(\Omega)$ is an unknown source term that we want to reconstruct based on the final time measurement $u(\cdot,T)$. To be specific, we will focus on a very practically physical scenario, where we assume that the observational measurement is collected point-wisely over a set of distributed sensors located at $\{x_i\}^n_{i=1}$ over the physical domain $\Om$ (see e.g.\cite{new2, new6, L08, NNR98}). We also assume the measurement data is always blurred by noise and takes the form $m_i=\mathcal{S}f^*(x_i)+e_i$, $i=1, \cdots, n$, where $f^*\in L^2(\Omega)$ is the true source term and $\{e_i\}^n_{i=1}$ are independent and identically distributed (i.i.d.) random variables on a proper probability space ($\mathfrak{X},\mathcal{F},\mathbb{P})$. From \cite{Chen-Zhang2021} and the analysis therein, we know that $\|u\|_{C([0,T];H^2(\Om))}\le C\|f\|_{L^2(\Om)}$. 
According to the embedding theorem of Sobolev spaces, we know that $H^2(\Omega)$ is continuously embedded into $C(\bar\Omega)$ so that  $\mathcal{S}f^*(x)$ is well defined point-wisely for all $x\in \Omega$. Without loss of generality, we assume that the scattered locations $\{x_i\}_{i=1}^n$ are uniformly distributed
in $\Omega$, i.e., there exists a constant $B>0$ such that ${d_{\max}}/{d_{\min}} \leq B$, where 
${d_{\max}}$ and ${d_{\min}}$ are defined by 
\begin{align}\label{aa}
d_{\max}=\mathop {\rm sup}\limits_{x\in \Omega} \mathop {\rm inf}\limits_{1 \leq i \leq n} |x-x_i|  
~~~\mbox{and} ~~ ~
d_{\min}=\mathop {\rm inf}\limits_{1 \leq i \neq j \leq n} |x_i-x_j|.
\end{align}
We denote the inner product between the measurement data and any function $v\in C(\bar\Omega)$ by $(m,v)_n=\frac{1}{n}\sum^n_{i=1}m_iv(x_i)$. Moreover, we denote the inner product between two functions by  $(u,v)_n=\frac{1}{n}\sum^n_{i=1}u(x_i)v(x_i)$ for any $u,v\in C(\bar\Omega)$ and the empirical norm $\|u\|_n=\big(\frac{1}{n}\sum_{i=1}^{n} u^2(x_i)\big)^{1/2}$ for any $u\in C(\bar\Omega)$. 

Equipped with these definitions, we can define the parabolic inverse source problem as to recover the unknown source term $f^*$ from the noisy final time measurement data $m_i=\mathcal{S}f^*(x_i)+e_i$, $i=1,...,n$. We will use the regularization method to solve this parabolic inverse source problem. Specifically, we will look for an approximate solution of the true source term $f^*$ by solving the following least-squares regularized minimization problem: 
\begin{align}\label{para-examp1}
\mathop {\rm min}\limits_{f\in X}         \|Sf-m\|^2_n+\lambda_n \|f\|_{L^2(\Omega)}^2\,,
\end{align}
where $\lambda_n$ is the regularization parameter. 

We first recall an existing work \cite{Chen-Zhang2021}, where the optimal stochastic convergence of regularized solutions and finite element solutions to time-dependent parabolic inverse source problems \eqref{zz0} have been studied. The following result represents the stochastic convergences corresponding to the random variables with bounded variance.
\begin{prop}[{Theorem 3.4 in \cite{Chen-Zhang2021}}]\label{para-examthm:2.1}
	Suppose $\{e_i\}^n_{i=1}$ are independent random variables satisfying
	$\mathbb{E}[e_i]=0$ and $\mathbb{E}[e^2_i]\leq \sigma^2$. Let $f_n\in L^2(\Omega)$ be the solution of the least-squares regularized minimization problem \eqref{para-examp1}. 
	Then there exist constants $\lambda_0 > 0$ and $C>0$ such that the following estimates hold 
	for any $0<\lambda_n \leq \lambda_0$:
	\begin{align}
	\mathbb{E} \big[\|\mathcal{S}f_n-\mathcal{S}f^*\|^2_n\big] \leq C \lambda_n \|f^*\|^2_{L^2(\Omega)} + \frac{C\sigma^2}{n\lambda^{d/4}_n}, \\
	\mathbb{E} \big[\|f_n-f^*\|^2_{L^2(\Omega)}\big] \leq C \|f^*\|^2_{L^2(\Omega)} + \frac{C\sigma^2}{n\lambda^{1+d/4}_n},\\
		\mathbb{E} \big[\|f_n-f^*\|^2_{H^{-1}(\Omega)}\big] \leq C \lambda_n^{1/2}\|f^*\|^2_{L^2(\Omega)} + \frac{C\sigma^2}{n\lambda^{1/2+d/4}_n}.
\end{align}
\end{prop}

A stronger stochastic convergence when $\{e_i\}^n_{i=1}$ are independent Gaussian random with variance $\sigma^2$ can also be found in \cite{Chen-Zhang2021}. Our convergence analysis below also applies to this case. From  Proposition \ref{para-examthm:2.1}, we have an optimal choice of the regularization parameter $\lambda_n$ in \eqref{para-examp1} as follows:
\begin{align}\label{opti-para}
\lam_n^{1/2+d/8}=O\big(\sigma n^{-1/2}\|f^*\|_{L^2(\Om)}^{-1}\big). 
\end{align} 
 It is an a priori estimate with knowing the noise level $\sigma$ and the knowledge of the true source term $f^*$, yet it is our goal to recover the true source term $f^*$ with unknown $\sigma$. The noise level is difficult to measure in many circumstances. In \cite{Chen-Zhang2021}, the authors proposed a self-consistent algorithm to compute the optimal $\lambda_n$ without any knowledge of the noise level $\sigma$ and the true source term $f^*$. Apply the POD method to this algorithm, we will also successfully determine the optimal parameter $\lambda_n$ iteratively. We find the proposed POD method maintains the same accuracy as the FEM.
Details can be found at the end of Section \ref{sec:eg1}. 

An effective way to approximate the optimal control problem \eqref{para-examp1} with a proper regularization parameter $\lambda_n$ is the iteration method. For each iteration, one has to solve the forward problem \eqref{zz0} and its adjoint problem for at least one time. In \cite{Chen-Zhang2021}, the FEM was used to approximate \eqref{para-examp1} and the optimal second-order convergence with respect to space has been proved. But as the DOF grows for a discrete method such as FEM and FDM, the iteration method will cost too much computational time. Therefore, we need to design numerical methods that allow us to efficiently and accurately solve \eqref{para-examp1}. 

%
%
%
%

\section{The data-driven model reduction method}\label{sec:DataDrivenModelReduction}
In this section, we will use the POD method to accelerate the inverse problem computation. We first construct the POD basis functions from the snapshot solutions of the parabolic equation \eqref{zz0} with some known type of source function. Then, we solve the optimization problem in the inverse problems with the constructed POD basis. 

 

\subsection{Construction of the POD basis functions}\label{sec:PODmethod}
\noindent
Assume that $u \in H^1_0(\Omega)$ is the solution to the following weak formulation of the parabolic equation \eqref{zz0}
\begin{align}\label{eqn:weak_formulation}
(\partial_t u, \psi) + a(u, \psi)  = (f, \psi), \quad \forall \psi \in H^1_0(\Omega), \quad t \in [0, T],  
\end{align}
where $\Omega \subset \mathbb{R}^d$ and $a(\cdot, \cdot)$ is a bilinear form on $H^1_0(\Omega) \times H^1_0(\Omega)$ that is defined according to the elliptic operator $\mathcal{L}$. Given a set of solutions at different time instances $\big\{ u(\cdot, t_0), u(\cdot, t_1), \ldots, u(\cdot, t_m) \big\}$ where $t_k = k \Delta t$ with $\Delta t = \frac{T}{m}$, we first get the solution snapshopts $\big\{ y_1, \ldots, y_{m + 1}, y_{m + 2}, \ldots, y_{2m + 1} \big\}$, where 
$y_k = u(\cdot, t_{k - 1})$, $k = 1, \ldots, m + 1$, and $y_k = \overline{\partial} u(\cdot, t_{k - m - 1})$, $k = m + 2, \ldots, 2m + 1$ with $\overline{\partial} u(\cdot, t_{k}) = \frac{u(\cdot, t_{k}) - u(\cdot, t_{k - 1})}{\Delta t}$, $k = 1, \ldots, m$.

Then, the POD basis functions $\{ \psi_k \}_{k = 1}^{{N_{pod}}}$ are built from the solution snapshopts by minimizing the following error: 
\begin{align}\label{POD_proj_err}
	\frac{1}{2m+1}\Big( \sum_{j = 0}^m \| u(t_j) - \sum_{k = 1}^{{N_{pod}}} (u(t_j), \psi_k)_{L^2(\Omega)} \psi_k \|_{L^2(\Omega)}^2 +	\sum_{j = 1}^m \| \overline{\partial} u(t_j) - \sum_{k = 1}^{{N_{pod}}} (\overline{\partial} u(t_j), \psi_k)_{L^2(\Omega)} \psi_k \|_{L^2(\Omega)}^2 \Big)
\end{align}
subject to the constraints that $\big(\psi_{k_1}(\cdot),\psi_{k_2}(\cdot)\big)_{L^2(\Omega)}=\delta_{k_1k_2}$, $1\leq k_1,k_2 \leq l$, where $\delta_{k_1k_2}=1$ if $k_1=k_2$, otherwise $\delta_{k_1k_2}=0$. 

Using the method of snapshot proposed by Sirovich \cite{Sirovich:1987}, we know that the optimization problem \eqref{POD_proj_err} can be reduced to an eigenvalue problem:
\begin{equation}\label{eigenvalueproblemPOD}
	Kv = \mu v,
\end{equation}
where the correlation matrix $K$ is computed from the solution snapshots $\{ y_1, y_2,\ldots, y_{2m + 1} \}$ with entries $K_{ij} = (y_i, y_j)_{L^2(\Omega)}$, $i,j = 1, \ldots, 2m + 1$, and is symmetric and semi-positive definite. 
We sort the eigenvalues in a decreasing order as $\mu_1 \geq \mu_2 \geq ... \geq \mu_{2m + 1} > 0$ and
the corresponding eigenvectors are denoted by $v_k$, $k=1,...,2m + 1$. It can be shown that if the POD basis functions are constructed by
\begin{equation}\label{PODbasisMethodOfSnapshot}
	\varphi_{k}(\cdot) = \frac{1}{\sqrt{\mu_k}}\sum_{j=1}^{2m + 1}(v_k)_j u(\cdot,t_j), \quad 1\leq k \leq 2m + 1,
\end{equation}
where $(v_k)_j$ is the $ j$-th  component of the eigenvector $v_k$, they will minimize the error \eqref{POD_proj_err}. This result as well as the error formula were proved in \cite{holmes:98}.

\begin{prop}[Sec. 3.3.2, \cite{holmes:98} or p. 502, \cite{Willcox2015PODsurvey}]\label{Prop_PODError}
	Let $\mu_1 \geq \mu_2 \geq ... \geq \mu_{2m + 1} > 0$ denote the positive eigenvalues of $K$ in \eqref{eigenvalueproblemPOD}. Then $\{\psi_k\}_{k=1}^{N_{pod}}$ constructed according to \eqref{PODbasisMethodOfSnapshot}
	is the set of POD basis functions, and we have the following error formula:
	\begin{equation}\label{SnapshotOfSolutions}		
		\frac{\sum_{i=1}^{2m + 1}\left|\left|y_i - \sum_{k=1}^{{N_{pod}}}\big(y_i,\psi_k(\cdot)\big)_{L^2(\Omega}\psi_k(\cdot)\right|\right|_{L^2(\Omega)}^{2}}{		 \sum_{i=1}^{2m+1}\left|\left|y_i\right|\right|_{L^2(\Omega)}^{2}} 
		= \frac{\sum_{k={N_{pod}}+1}^{2m + 1}  \mu_k}{\sum_{k=1}^{2m + 1}  \mu_k},
	\end{equation} 
	where the number ${N_{pod}}$ is determined according to the ratio $\rho=\frac{\sum_{k={N_{pod}}+1}^{2m + 1}  \mu_k}{\sum_{k=1}^{2m + 1}\mu_k}$. 
\end{prop}
In practice, we shall make use of the decay property of eigenvalues in $\mu_k$ and choose the first $m$ dominant eigenvalues such that the ratio $\rho$ is small enough to achieve an expected accuracy, for instance $\rho=1\%$.  One would prefer the eigenvalues decays as fast as possible so that one can ensure high accuracy with few POD basis functions.

\subsection{A Fast algorithm for solving parabolic inverse source problems}



We now propose the fast algorithm for solving parabolic inverse source problems based on some given discretization, including our proposed POD method.  We first define the functional $\mathcal{J}$ to be 
\begin{align}
	\mathcal{J}[f]=\|\mathcal{S}f-m\|^2_n+\lambda_n \|f\|_{L^2(\Omega)}^2.
\end{align}	 
Then, the least-squares regularized minimization problem \eqref{para-examp1} becomes to solve the following misfit functional:
\begin{align}
	\mathop {\rm min}\limits_{f\in L^2(\Omega)}  \mathcal{J}[f].
\end{align}
\begin{lemma}\label{Fdifferentiable}
	The misfit functional $\mathcal{J}[f]$ is Fr$\acute{e}$chet-differentiable. 
\end{lemma}
\begin{proof} 
	From the definition of Fr$\acute{e}$chet-differentiable, we need to compute
	\begin{align}\label{Frechet-differentiable}
		d\mathcal{J}[f](v)&=\lim_{t\rightarrow 0} \frac{\mathcal{J}[f+tv]-\mathcal{J}[f]}{t}\nonumber\\
		&= (\mathcal{S}f-m,\mathcal{S}v)_n + \lambda_n(f,v)\nonumber\\
		&= \big(\mathcal{S}^*(\mathcal{S}f-m),v\big) + \lambda_n(f,v)\nonumber\\
		&= \big(\mathcal{S}^*(\mathcal{S}f-m)+ \lambda_n f,v\big), 
	\end{align}
	where $v\in L^2(\Omega)$. In \eqref{Frechet-differentiable}, the second equality is easily to get from the quadratic form of the misfit functional $\mathcal{J}[f]$ and $\mathcal{S}^*$ is the adjoint operator of $\mathcal{S}$ in the third equality. Thus, one can directly obtain that 
	\begin{align}
		d\mathcal{J}[f]=\mathcal{S}^*(\mathcal{S}f-m)+ \lambda_n f. \label{dJf}
	\end{align}
\end{proof}
The formula \eqref{dJf} in Lemma \ref{Fdifferentiable} allows us to apply the  gradient descent method to minimize
the discrepancy functional $\mathcal{J}[f]$. Let $f_{0}$ be the initial guess and $f_k$ denote the the solution of the least-squares regularized minimization problem \eqref{para-examp1} at the $k$-th iteration step, we update the iterative solution by 
\begin{align}
	f_{k+1}=f_{k}-\alpha d\mathcal{J}[f_{k}], \quad \forall k\in \mathbb{N},
\end{align}
where $\alpha$ is the step size.

Given a fully discrete scheme with solution space $V_{dis}$, we define an approximate forward operator $\mathcal{S}_{dis}$: $L^2(\Omega)\rightarrow V_{dis}$, such that, given any source function $f$, $\mathcal{S}_{dis}f$ gives the numerical solution at final time. In such discrete setting,  we turn to solve the following misfit functional:
\begin{align}
	\mathop {\rm min}\limits_{f\in \Psi}  \mathcal{J}_{dis}[f],
\end{align}
where the functional $\mathcal{J}_{dis}[f]=\|\mathcal{S}_{dis}f-m\|^2_n+\lambda_n \|f\|_{L^2(\Omega)}^2$. We can compute the 
Fr$\acute{e}$chet derivative of $\mathcal{J}_{dis}[f]$ and obtain the following iterative scheme:
\begin{align}\label{eqn:pod_iteration}
	f_{k+1}=f_{k}-\alpha d\mathcal{J}_{dis}[f_{k}], \quad \forall k\in \mathbb{N},
\end{align}
where $\alpha$ is the step size, $d\mathcal{J}_{dis}[f]=\mathcal{S}_{dis}^*(\mathcal{S}_{dis}f-m)+ \lambda_n f$, and $f_{0}$
is some initial guess. 

The iterative scheme \eqref{eqn:pod_iteration} works with both FEM-based Galerkin method and POD-based Galerkin method. However, we emphasize that the DOF of the POD basis functions is much smaller than that of the FEM space. Thus, the POD-based Galerkin method can provide significant computational savings in solving the parabolic equation \eqref{zz0} than the FEM, which allows us to quickly compute the iterative scheme \eqref{eqn:pod_iteration}. Therefore, we obtain a fast algorithm in solving parabolic inverse source problem.  

\subsection{Complete algorithm}\label{sec:complete_algorthm}
\noindent  Due to the nature of the inverse source problem, the source term is unknown. We assume it to be a sample from some random space. Such random space may not have a closed form or 
a finite-dimensional parameterization, so we further 
assume that we have $N_f$ realizations, $\{f_l\}_{l=1,\cdots, N_f}$ of the source term. In other words, $\{f_l\}_{l=1,\cdots, N_f}$ are some possible ground truth source function in the inverse problem. Hence, deviating from the classic POD algorithm in solving parabolic equation introduced in Section \ref{sec:PODmethod}, we include snapshots from each proposed source function into the minimization \eqref{POD_proj_err}. Estimation in \eqref{SnapshotOfSolutions} still holds as it is only an algebraic result. 

Finally, we summarize the proposed data-driven model reduction method for solving parabolic inverse source problems in Algorithm \ref{alg:InversePOD}, where the notations have been defined before.  
\begin{algorithm}[h]
	\SetAlgoLined
	\textbf{Input}: Observation data $m$. Proposed source function $\{f_l\}_{l=1,\cdots, N_f}$, error thresholds for constructing POD basis functions $\epsilon>0$ and for the optimization problem $tol$, computational time $T$, time step of the FEM $\Delta t$, $N_T=\lceil T/\Delta t\rceil$, mesh size of the FEM $h$, one step finite element solver for the forward problem $\mathcal{O}_{FEM}$, where  $\mathcal{S}_{FEM}f=(\mathcal{O}_{FEM}(f))^{N_T}u_0$, and step size in the gradient descent method  $\alpha$.\\
	\For{$l=1:N_f$}{
		Solve the forward problem and store the snapshots $(\mathcal{O}_{FEM}(f_l))^{i}u_0$ for $i=1,\cdots, N_T$.\\
	}
	{Concatenate} all snapshots as $\{S_j\}_{j=1,\cdots, N_S=N_T\times N_f}=\cup_{i=1,\cdots, N_T, l=1,\cdots, N_f} \{(\mathcal{O}_{FEM}(f_l))^{i}u_0\}$;\\
	compute covariance matrix $K=(K_{ij})$ where $K_{ij}=\frac{1}{N_S} (S_i, S_j)_{L^2(\Omega)}$;\\
	compute the SVD of $K=(K_{ij})$ and denote the eigenvalues in a decreasing order as $\mu_1 \geq \mu_2 \geq ... \geq \mu_{N_S} > 0$ and the corresponding eigenvectors are denoted by $v_k$;\\
	find minimal $N_{pod}$ such that $\frac{\sum_{k=N_{pod}+1}^{N_S}  \mu_k}{\sum_{k=1}^{N_S}  \mu_k}<\epsilon$;\\
	construct POD basis by $\varphi_{k}(\cdot) = \frac{1}{\sqrt{\mu_k}}\sum_{j=1}^{N_S}(v_k)_j S_j, \quad 1\leq k \leq N_{pod}$;\\
	construct Galerkin type solver for forward problem on POD basis $\{\varphi_{k}\}_{k=1,\cdots, N_{pod}}$ as $S_{POD}$;\\
	Set $f$ as an initial guess $f=f_0$.\\
	\While{$\mathcal{J}_{pod}[f_k]=\|\mathcal{S}_{pod}f_k-m\|^2_n+\lambda_n \|f_k\|_{L^2(\Omega)}^2>tol$ }{
		update $f_{k+1}\leftarrow f_k-\alpha d\mathcal{J}_{pod}[f_{k}]$ where  $d\mathcal{J}_{pod}[f_k]=\mathcal{S}_{pod}^*(\mathcal{S}_{pod}f_k-m)+ \lambda_n f_k$.
	}
	\textbf{Output}: computed source term $f$.
	\caption{\textbf{A fast algorithm for solving parabolic inverse source problem}}
	\label{alg:InversePOD}
\end{algorithm}

\section{Convergence analysis}\label{sec:ConvergenceAnalysis}
In this section, we will first present the general discrete approximation to the optimization problem \eqref{para-examp1}. 
Recall that $V_{dis}\subset X$ be the discrete function space with dimension $N_{dis}$ and $\mathcal{S}_{dis}$ be the discrete approximation of the operator $\mathcal{S}:L^2(\Omega)\to H^2(\Omega)$. We make the following assumptions on the discrete function space $V_{dis}$ and the discretized operator $\mathcal{S}_{dis}$.

\begin{assu}\label{Assumption3}
	(1) There exists an error estimate term $e_{dis}$ such that the discrete operator $S_{dis}$ satisfies 
\begin{align}
	\|\mathcal{S}f-\mathcal{S}_{dis}f\|^2_n\leq Ce_{dis}\|f\|^2_{L^2(\Omega)}.
\end{align}
	(2) For any $f\in L^2(\Omega)$, there exists $v\in V_{dis}$ such that,
	\begin{align}
	\lambda_n \|f-v\|^2_{L^2(\Omega)} + \|\mathcal{S}_{dis}(f-v)\|^2_n\leq C(\lambda_n+e_{dis})\|f\|^2_{L^2(\Omega)}.
\end{align}
\end{assu}

{In the general discrete approximation  of the problem \eqref{para-examp1}}, we will solve the least-squares regularized minimization problem:
\begin{align}\label{disc}
\mathop {\rm min}\limits_{f_{dis}\in V_{dis}}         \|\mathcal{S}_{dis}f_{dis}-m\|^2_n+\lambda_n \|f_{dis}\|_{L^2(\Omega)}^2.
\end{align}
The following proposition gives the convergence analysis for the discretized optimization schemes in solving the parabolic inverse source problems, which is proved in \cite{Chen-Zhang2021}.

\begin{prop}[Theorem 3.10 in \cite{Chen-Zhang2021}]\label{thm:4.1}
	Let $\{e_i\}^n_{i=1}$ are independent random variables satisfying $\mathbb{E}[e_i]=0$ and $\mathbb{E}[e^2_i]\leq \sigma^2$ for $i=1,\cdots, n$ and $f_{dis}$ be the solution of \eqref{disc}.
	Then there exist constants $\lambda_0 > 0$ and $C>0$, such that for any $\lambda_n \leq \lambda_0$  the following estimates hold true:
	\begin{align}\label{d2}
	\mathbb{E}\big[\|\mathcal{S}f^*- \mathcal{S}_{dis}f_{dis}\|_n^2\big]\le C(\lambda_n+e_{dis})\|f^*\|^2_{L^2(\Omega)}+C\left(1+\frac{e_{dis}}{\lam_n}+\frac{N_{dis} e_{dis}}{\lambda^{1-d/4}_n}\right)\frac{\sigma^2}{n\lambda_n^{d/4}},
\end{align}
	\begin{align}
	\mathbb{E}\big[\|f^*- f_{dis}\|_{L^2(\Omega)}^2\big]\le C\frac{\lambda_n+e_{dis}}{\lambda_n}\|f^*\|^2_{L^2(\Omega)}+C\left(1+\frac{e_{dis}}{\lam_n}+\frac{N_{dis} e_{dis}}{\lambda^{1-d/4}_n}\right)\frac{\sigma^2}{n\lambda_n^{1+d/4}},
\end{align}
and
	\begin{align}
	\mathbb{E}\big[\|f^*- f_{dis}\|_{H^{-1}(\Omega)}^2\big]&\le C(\lambda^{1/2}_n+e^{1/2}_{dis})\frac{\lambda_n+e_{dis}}{\lambda_n}\|f^*\|^2_{L^2(\Omega)}
	\nonumber\\
	&+C(\lambda^{1/2}_n+e^{1/2}_{dis})\left(1+\frac{e_{dis}}{\lam_n}+\frac{N_{dis} e_{dis}}{\lambda^{1-d/4}_n}\right)\frac{\sigma^2}{n\lambda_n^{1+d/4}}.
\end{align}
\end{prop}
Now we study the error analysis of the proposed method in solving the parabolic inverse source problem, where we used the backward Euler scheme to discretize the temporal derivative. Let $u$ be the solution  to \eqref{eqn:weak_formulation} and $\hat{u}$ is the solution computed by semi-discrete scheme based on FEM basis functions. {We also assume the POD basis functions are constructed by snapshots of $\hat{u}$ and their temporal finite differences}. The FEM space and POD space are denoted by $V_h$ and $\Psi$, respectively. The fully discrete scheme is constructed on $\Psi$ and the solution is denoted by $U_k$ for $k=1\cdots m$ with $ m=\frac{T}{\Delta t}$. To be precise, we seek $U_k$ such that 
{ \begin{equation}\label{eqn:fully_discrete_POD}
 	(\bar{\partial}U_k,\psi)+a(U_k,\psi)=(f,\psi), \quad \forall \psi\in \Psi.
 \end{equation}
 }
Now we denote the solution operator from source term $f$ to the final time solution $U_m$ as $\mathcal{S}_{pod}$, such that $\mathcal{S}_{pod}f=U_m$.
\begin{theorem}\label{thm_convergence}
	Let $f_{pod}$ be the solution of \eqref{disc}, then there exist constants $\lambda_0 > 0$ and $C>0$, such that for any $\lambda_n \leq \lambda_0$ the following estimates hold true:
	\begin{align}\label{pod1}
		\mathbb{E}\big[\|\mathcal{S}f^*- \mathcal{S}_{pod}f_{pod}\|_n^2\big]\le C(\lambda_n+e_{pod})\|f^*\|^2_{L^2(\Omega)}+C\left(1+\frac{e_{pod}}{\lam_n}+\frac{{N_{pod}} e_{pod}}{\lambda^{1-d/4}_n}\right)\frac{\sigma^2}{n\lambda_n^{d/4}},
	\end{align}
	\begin{align}
		\mathbb{E}\big[\|f^*- f_{pod}\|_{L^2(\Omega)}^2\big]\le C\frac{\lambda_n+e_{pod}}{\lambda_n}\|f^*\|^2_{L^2(\Omega)}+C\left(1+\frac{e_{pod}}{\lam_n}+\frac{{N_{pod}} e_{pod}}{\lambda^{1-d/4}_n}\right)\frac{\sigma^2}{n\lambda_n^{1+d/4}},
	\end{align}
	and
	\begin{align}
	\mathbb{E}\big[\|f^*- f_{pod}\|_{H^{-1}(\Omega)}^2\big]&\le C(\lambda^{1/2}_n+e^{1/2}_{pod})\frac{\lambda_n+e_{pod}}{\lambda_n}\|f^*\|^2_{L^2(\Omega)} \nonumber \\
	&+C(\lambda^{1/2}_n+e^{1/2}_{pod})\left(1+\frac{e_{pod}}{\lam_n}+\frac{{N_{pod}} e_{pod}}{\lambda^{1-d/4}_n}\right)\frac{\sigma^2}{n\lambda_n^{1+d/4}},
	\end{align}
	where 
	\begin{equation}
	    e_{pod}=\left(h^2+\Delta t|\ln(\Delta t)|+\sqrt{\frac{T}{\Delta t}\rho}\right)^2.
	\end{equation}
\end{theorem}
\begin{proof}
To simplify the notation in the proof, let $\|\cdot\|$ denote the norm in $L^2(\Omega)$, unless otherwise specified. For $k=1,..., m$ , we consider the decomposition as follows:
\begin{align}
	U_{k}-\hat{u}\left(t_{k}\right)=U_{k}-\mathcal{P}^{{N_{pod}}} \hat{u}\left(t_{k}\right)+\mathcal{P}^{{N_{pod}}} \hat{u}\left(t_{k}\right)-\hat{u}\left(t_{k}\right)=\vartheta_{k}+\varrho_{k},
\end{align}
where $\vartheta_{k}=U_{k}-\mathcal{P}^{{N_{pod}}} \hat{u}\left(t_{k}\right)$, $\varrho_{k}=\mathcal{P}^{{N_{pod}}} \hat{u}\left(t_{k}\right)-\hat{u}\left(t_{k}\right)$, and $\mathcal{P}^{{N_{pod}}}$ denotes the  Ritz-projection to $\Psi$ under the bilinear form $a(\cdot,\cdot)$. The projection error $\varrho_{m}$ can be bounded by \eqref{SnapshotOfSolutions}. In addition, by applying classic result of parabolic equation, i.e.,  $\| u(\cdot,t_j)\|\leq C\| f\|$, $\forall j=1,\cdots,m$ and $\| u_t(\cdot,t_j)\|\leq C\|  f\|$, $\forall j=1,\cdots,m$, we know that, 
\begin{align}\label{est:rho_m}
	\varrho_{m}\leq C \sqrt{(2m+1)\rho}\| f\|.
\end{align}

Now we only need to estimate the term $\vartheta_{m}$. According to  \eqref{eqn:fully_discrete_POD}, we have 
\begin{align}\label{eqn:theta_weak_formula}
	\begin{aligned}
		\left(\bar{\partial} \vartheta_{k}, \psi\right)+a\left(\vartheta_{k}, \psi\right)=&\left(\bar{\partial} U_{k}, \psi\right)+a\left(U_{k}, \psi\right) -\left(\bar{\partial} \mathcal{P}^{{N_{pod}}} \hat{u}\left(t_{k}\right), \psi\right)-a\left(\mathcal{P}^{{N_{pod}}} \hat{u}\left(t_{k}\right), \psi\right) \\
		=&\left(f\left(t_{k}\right), \psi\right)-\left(\bar{\partial} \mathcal{P}^{{N_{pod}}} \hat{u}\left(t_{k}\right), \psi\right)-a\left(\hat{u}\left(t_{k}\right), \psi\right) \\
		=&\left(v_{k}, \psi\right),
	\end{aligned}
\end{align}
where 
\begin{align*}
	v_{k}=\hat{u}_{t}\left(t_{k}\right)-\bar{\partial} \mathcal{P}^{{N_{pod}}} \hat{u}\left(t_{k}\right)=\hat{u}_{t}\left(t_{k}\right)-\bar{\partial} \hat{u}\left(t_{k}\right)+\bar{\partial} \hat{u}\left(t_{k}\right)-\bar{\partial} P^{{N_{pod}}} \hat{u}\left(t_{k}\right) .
\end{align*}
We define $w_{k}=\hat{u}_{t}\left(t_{k}\right)-\bar{\partial} \hat{u}\left(t_{k}\right)$ and $z_{k}=\bar{\partial} \hat{u}\left(t_{k}\right)-\mathcal{P}^{{N_{pod}}} \bar{\partial} \hat{u}\left(t_{k}\right)$. Substituting  $\psi=\vartheta_{k} $ into \eqref{eqn:theta_weak_formula}, we derive that
\begin{align}
	\left\|\vartheta_{k}\right\|^{2}-\left(\vartheta_{k}, \vartheta_{k-1}\right)+\Delta t a\left(\vartheta_{k}, \vartheta_{k}\right) \leq \Delta t\left\|v_{k}\right\|\left\|\vartheta_{k}\right\|
\end{align}
and therefore,
\begin{align}
	\left\|\vartheta_{k}\right\| \leq \frac{1}{1+\frac{a_{min}}{\alpha} \Delta t}\left(\left\|\vartheta_{k-1}\right\|+\Delta t\left\|v_{k}\right\|\right),
\end{align}
where $a_{min}$ is the minimum of the coefficient $a(x)$ in the elliptic operator and $\alpha$ is a constant such that 
\begin{align}
	\|\phi\|^2_{L_2}\leq \alpha \|\phi\|_{H^1} \quad \forall \phi\in V_h.
\end{align}
Taking $\gamma=\frac{a_{min}}{\alpha}$, we have 
\begin{align}\label{eqn:est_vartheta1}
	\left\|\vartheta_{m}\right\| \leq\left(\frac{1}{1+\gamma \Delta t}\right)^{m}\left\|\vartheta_{0}\right\|+\Delta t \sum_{j=1}^{m}\left(\frac{1}{1+\gamma \Delta t}\right)^{m-j+1}\left\|v_{j}\right\|.
\end{align}
As $\| v_j\|\leq\| w_k\|+\| z_k\|$, so \eqref{eqn:est_vartheta1} is taken into two parts. 
\begin{align}
	\Delta t \sum_{j=1}^{m}\xi^{m-j+1}\left\|z_{j}\right\|\leq&\Delta t \sqrt{(\sum_{j=1}^m\xi^{2(m-j+1)})(\sum_{j=1}^m\| z_k\|^2)},
\end{align}
where $\xi$ denotes $\frac{1}{1+\gamma\Delta t}$. Using the facts that $\left(\frac{1}{\xi}\right)^{2 m}=(1+\gamma \Delta t)^{2 m}=\left(1+\frac{2 \gamma T}{2 m}\right)^{2 m} \leq e^{2 \gamma T}$ and $\xi^{-2}-1=(1+\gamma\Delta t)^2-1\geq 2\gamma\Delta t$, we can obtain that 
\begin{align}
\Delta t \sqrt{(\sum_{j=1}^m\xi^{2(m-j+1)})(\sum_{j=1}^m\| z_k\|^2)}=&T\sqrt{(\frac{1}{m}\sum_{j=1}^m\xi^{2(m-j+1)})(\frac{1}{m}\sum_{j=1}^m\| z_j\|^2)}\nonumber\\
\leq&T\sqrt{(\frac{1}{m}\frac{1-\xi^{2m}}{\xi^{-2}-1})(\frac{1}{m}\sum_{j=1}^m\| z_j\|^2)}\nonumber\\
\leq&T\sqrt{\frac{1-e^{-2\gamma T}}{2\gamma T}(\frac{1}{m}\sum_{j=1}^m\| z_j\|^2)}.
\end{align}
By Corollary 4 in \cite{kunisch2001galerkin}, we have that $	\frac{1}{m}\sum_{j=1}^m\| z_j\|^2\leq C\rho\| f\|^2$, which implies that 
\begin{align}\label{est:zj}
	\Delta t \sum_{j=1}^{m}\xi^{m-j+1}\left\|z_{j}\right\|\leq C(T)\sqrt{\rho}\| f\|.
	\end{align}
Now we turn to estimate the term $\Delta t \sum_{j=1}^{m}\xi^{m-j+1}\left\|w_{j}\right\|$. By definition,
\begin{align}
	\begin{aligned}
		w_{j} &=u_{t}\left(t_{j}\right)-\frac{u\left(t_{j}\right)-u\left(t_{j-1}\right)}{\Delta t} \\
		&=\frac{1}{\Delta t}\left(\Delta t u_{t}\left(t_{j}\right)-\left(u\left(t_{j}\right)-u\left(t_{j-1}\right)\right)\right) \\
		&=\frac{1}{\Delta t} \int_{t_{j-1}}^{t_{j}}\left(s-t_{j-1}\right) u_{t t}(s) d s.
	\end{aligned}
\end{align}
By the analytic semigroup theory \cite{renardy2006introduction} or Lemma 3.6 in \cite{Chen-Zhang2021}, we have the following regularity estimate,  
\begin{align}
	\| u_{tt}(s)\|\leq\frac{C}{s}\| f\|,
\end{align}
which implies,
\begin{align}
	\| w_j\|\leq\frac{1}{\Delta t} \int_{t_{k-1}}^{t_{j}}\left(s-t_{k-1}\right)\frac{1}{s}\| f\| ds.
	\end{align}
When $j=1$, 
\begin{align}
	\| w_1\|\leq\frac{1}{\Delta t} \int_{0}^{\Delta t}\| f\| ds=\| f\|,
\end{align}
and when $j\geq 2$,
\begin{align}
	\| w_k\|\leq \frac{1}{\Delta t} \int_{t_{j-1}}^{t_{j}}\left(s-t_{j-1}\right)\frac{1}{t_{j-1}}\| f \| ds = \frac{1}{2}\| f\|\frac{1}{(j-1)}\leq\frac{\| f \|}{j}.
\end{align}
Therefore, we obtain that 
\begin{align}
	\Delta t \sum_{j=1}^{m}\xi^{m-j+1}\left\|w_{j}\right\| \leq   	\Delta t \sum_{j=1}^{m}\xi^{m-j+1} \frac{\| f\|}{j} 
	\leq   \Delta t \sum_{j=1}^{m} \frac{\| f\|}{j} \leq C\Delta t |\ln(\Delta t)| \| f \|\label{est:wj}.
\end{align}
Summing up \eqref{est:zj}, \eqref{est:wj} and using the classic parabolic FEM theory (e.g. Theorem 1.2 in \cite{thomee1990finite}), we have that 
\begin{align}\label{est:parabolic_FEMsemi}
	\| \hat{u}\left(t_{m}\right)-u\left(t_{m}\right)\|\leq Ch^2\| f \|.
\end{align} 
Therefore, we can prove that
\begin{align}
	\| U_m-u(T)\|\leq C\left(h^2+\Delta t|\ln(\Delta t)|+\sqrt{\frac{T}{\Delta t}\rho}\right)\| f \|.
\end{align}
Finally,  we apply the Proposition \ref{thm:4.1} by taking $\mathcal{S}_{pod}$ as $\mathcal{S}_{dis}$ and $f_{pod}$ as $f_{dis}$, and prove the estimates in the Theorem \ref{thm_convergence}. Clearly, $e_{pod}=\left(h^2+\Delta t|\ln(\Delta t)|+\sqrt{\frac{T}{\Delta t}\rho}\right)^2$.
\end{proof}

\begin{rem}
Theorem \ref{thm_convergence} suggests that the optimal smoothing parameter in Eq.\eqref{opti-para} may still apply. In particular, if $e_{pod}\le C\lambda_n$ and ${N_{pod}} e_{pod}\leq C\lambda_n^{1-1/\alpha}$, we have
	\begin{align}\label{pod2}
		\mathbb{E}\big[\|\mathcal{S}f^*- \mathcal{S}_{pod}f_{pod}\|_n^2\big]\le C\lambda_n\|f^*\|_{L^2(\Omega)}^2+\frac{C\sigma^2}{n\lambda_n^{d/4}},
	\end{align}
	\begin{align}
		\mathbb{E}\big[\|f^*- f_{pod}\|_{L^2(\Omega)}^2\big]\le C\|f^*\|_{L^2(\Omega)}^2+\frac{C\sigma^2}{n\lambda_n^{1+d/4}},
	\end{align}
	and
		\begin{align}
		\mathbb{E}\big[\|f^*- f_{pod}\|_{H^{-1}(\Omega)}^2\big]\le C\lambda^{1/2}_n\|f^*\|_{L^2(\Omega)}^2+\frac{C\sigma^2}{n\lambda_n^{1/2+d/4}}.
	\end{align}
	Note that $\rho$ depends on the number of POD basis functions, ${N_{pod}}$, and quickly decreases to zero when increasing ${N_{pod}}$. Assumptions above are achievable by taking $h$ and $\Delta t$ to be small and ${N_{pod}}$ to be small with respect to the DOF of the FEM.

\end{rem} 
\section{Numerical examples}\label{num-sec}
\noindent
In this section, we present numerical examples to investigate the performance of the proposed POD method in solving parabolic inverse source problems. Two types of source functions, i.e., letters and circles at different locations, will be studied. In the example of letters, we verify the convergence introduced in Theorem \ref{thm_convergence} and apply the iterative method to find optimal smoothing parameters. In the second example, we mainly study the approximation property of our POD basis functions.

\subsection{Recovering Letters}\label{sec:eg1}

\paragraph{General Settings} For the first example, we apply the POD method to recover the source term $f$ in form of discontinuous patterns.  To be precise, $f$ is an indicator function, with the maximum value $1$, of a set in $[0,1]^2$ domain whose shape is decided by a capital letter. For each observation data, we first apply the FEM with $h=1/32$ and time step $\Delta t=1/32$ to get the solution of the forward problem with exact source term $f$ at final time $T=1$. We interpolate the FEM solution to locations of $100\times 100$ evenly distributed sensor and add a Gaussian unbiased noise with stand derivation $10^{-3}$ to each observation. It corresponds to a $10\%$ noise level. Initial value $g=0$ in \eqref{zz0}.

To construct the POD basis functions, we generate the snapshots source term decided by letters from $A$ to $O$. The snapshots are generated by solving the forward problem with each source term via FEM (space mesh size $h=1/32$, time step $\Delta t=1/32$). Under such settings,  $21$ POD basis functions are constructed to resolve the snapshots of the forward problem with $10^{-4}$ accuracy. In other words, the POD basis functions cover $1-10^{-4}$ of eigenspace from the snapshots with respect to the $L^2$ norm. To solve the inverse problem, $\lambda$ is set to be  $10^{-6}$ and we will discuss how to find the optimal $\lambda$ later. 
\paragraph{Computational Cost} In Figure \ref{fig:eg1_multi_letters}, we show some of the letters recovered. The computational cost of the POD method includes two parts. First, we will construct the POD basis functions for all the $15$ letters, which will take about $40.3$ seconds on a  desktop computer (3.1GHz, 6-core, i5, Matlab). Then, for each letter, it will take about $164$  seconds to solve the optimization problem via gradient descend (relative $L_2$ threshold $10^{-5}$) using the POD basis functions. As a comparison, with the FEM basis functions, it takes $1328.4$ seconds to recover similar patterns with the same configuration during optimization. In the case of recovering one single letter, the POD method has increased the speed by $6.5\times$, and in the case of recovering all $15$ letters, it achieves about $8\times$ speed-up. 

\begin{figure}[htb]
	\centering
	\subfigure[A]{\includegraphics[width=0.32\linewidth]{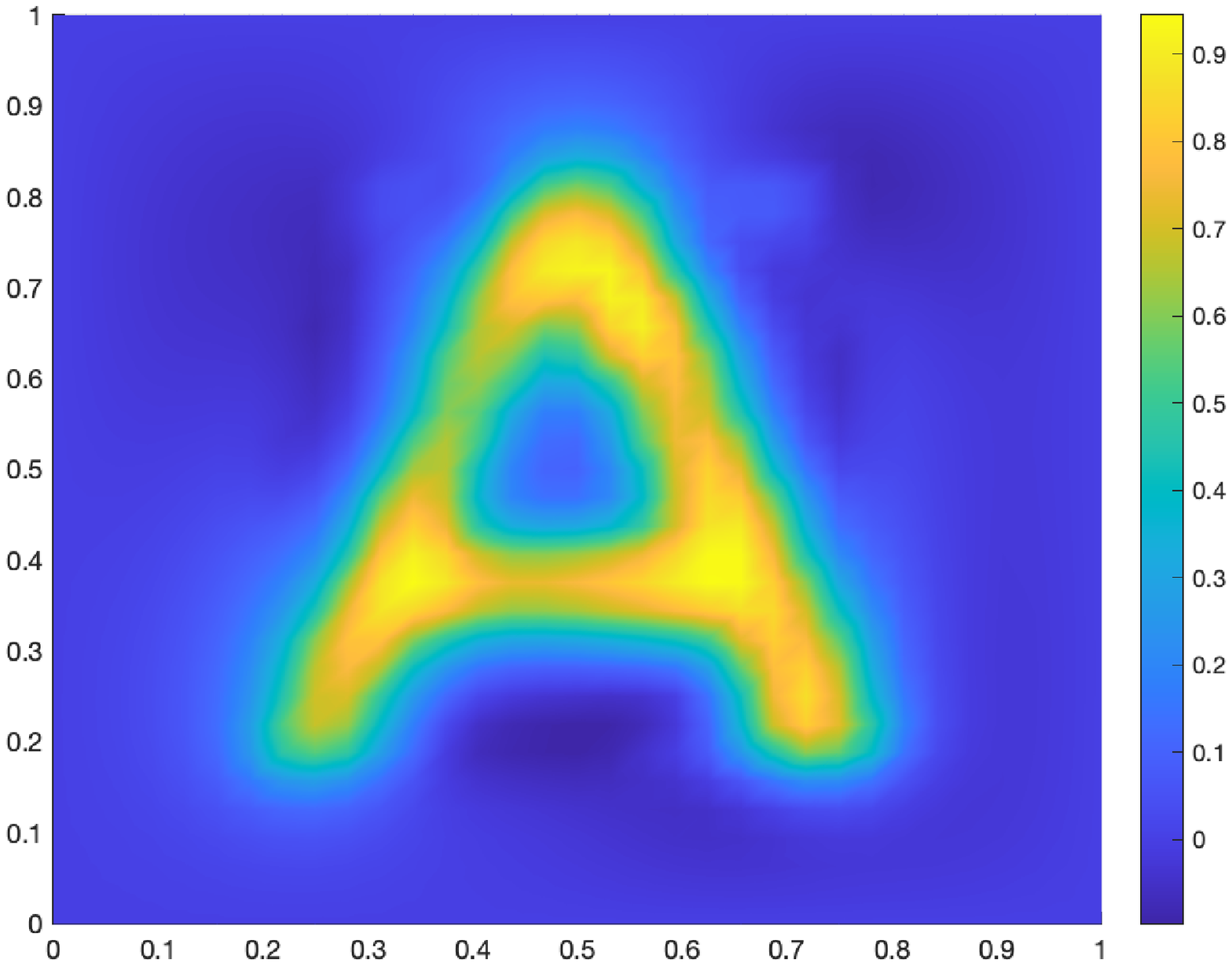}}
	\subfigure[C]{\includegraphics[width=0.32\linewidth]{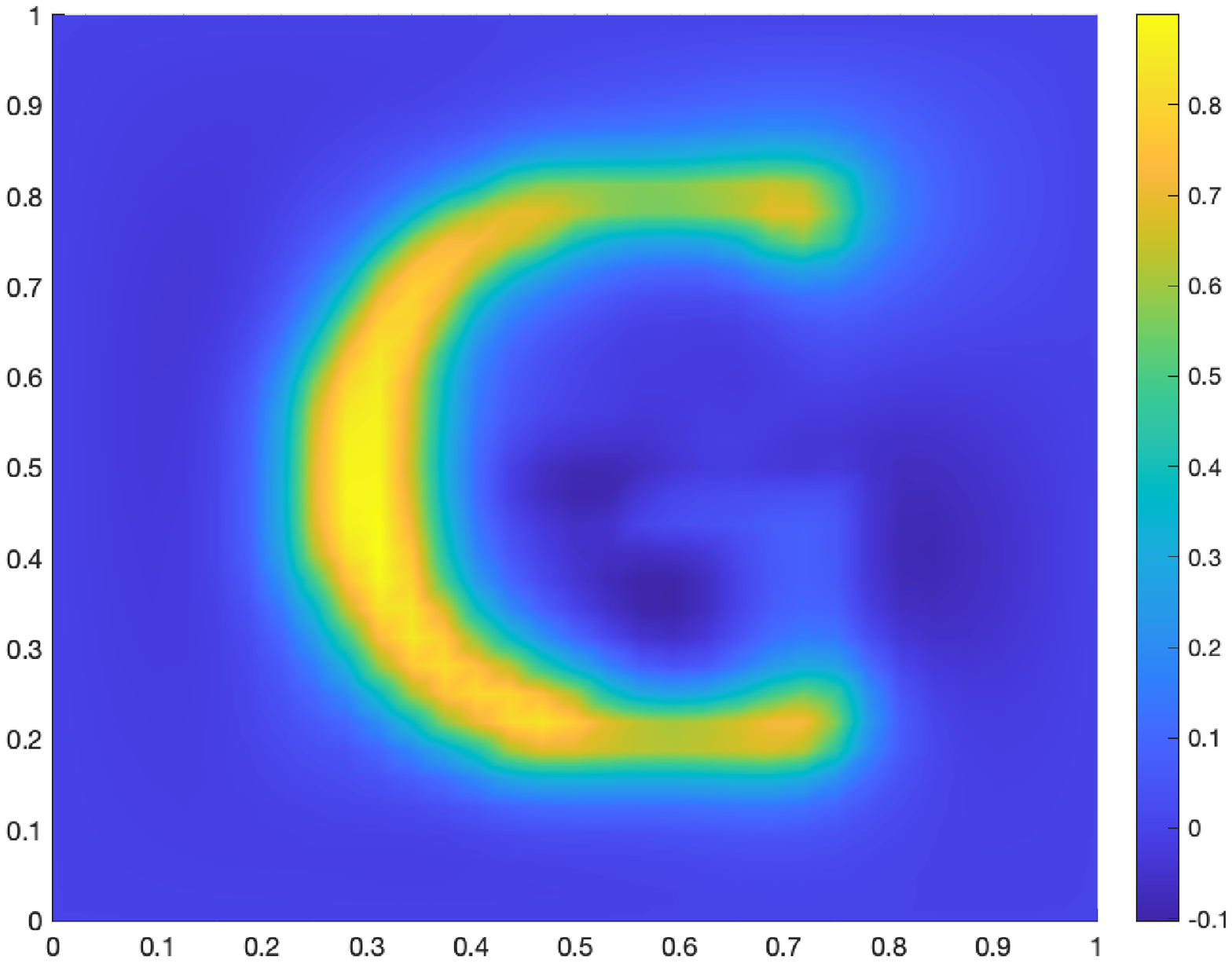}}	
	\subfigure[G]{\includegraphics[width=0.32\linewidth]{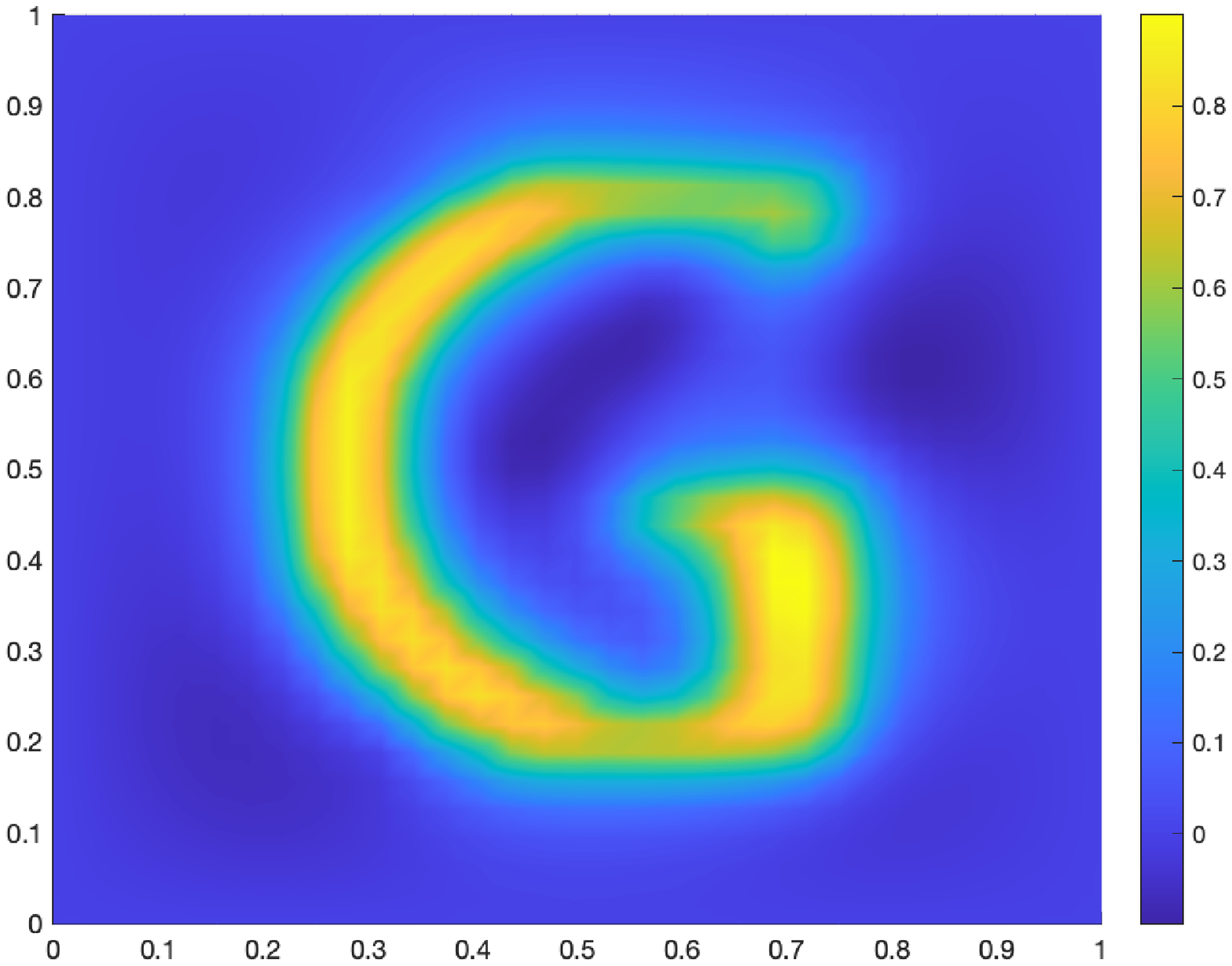}}
	\subfigure[reference A]{\includegraphics[width=0.32\linewidth]{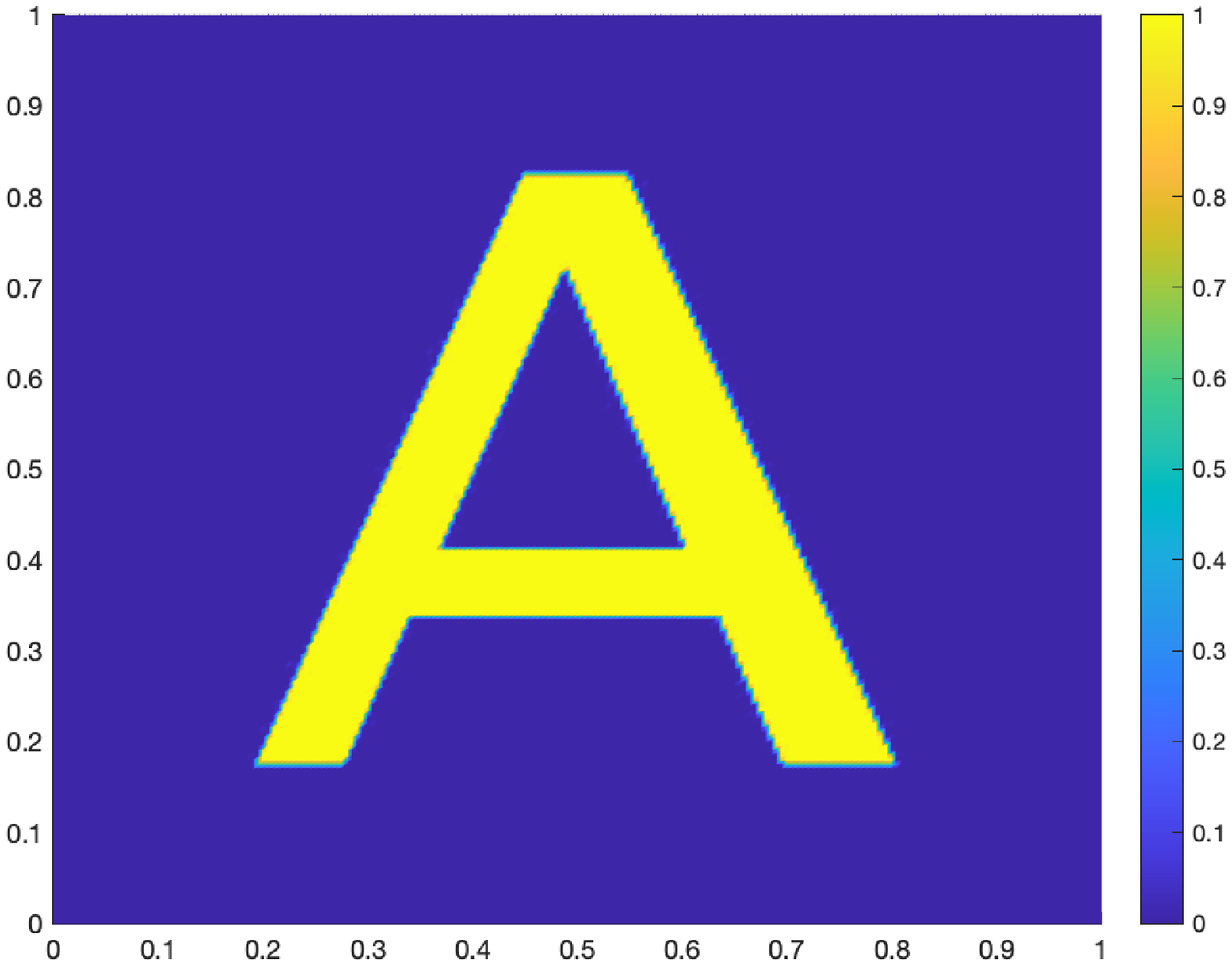}}		\subfigure[D]{\includegraphics[width=0.32\linewidth]{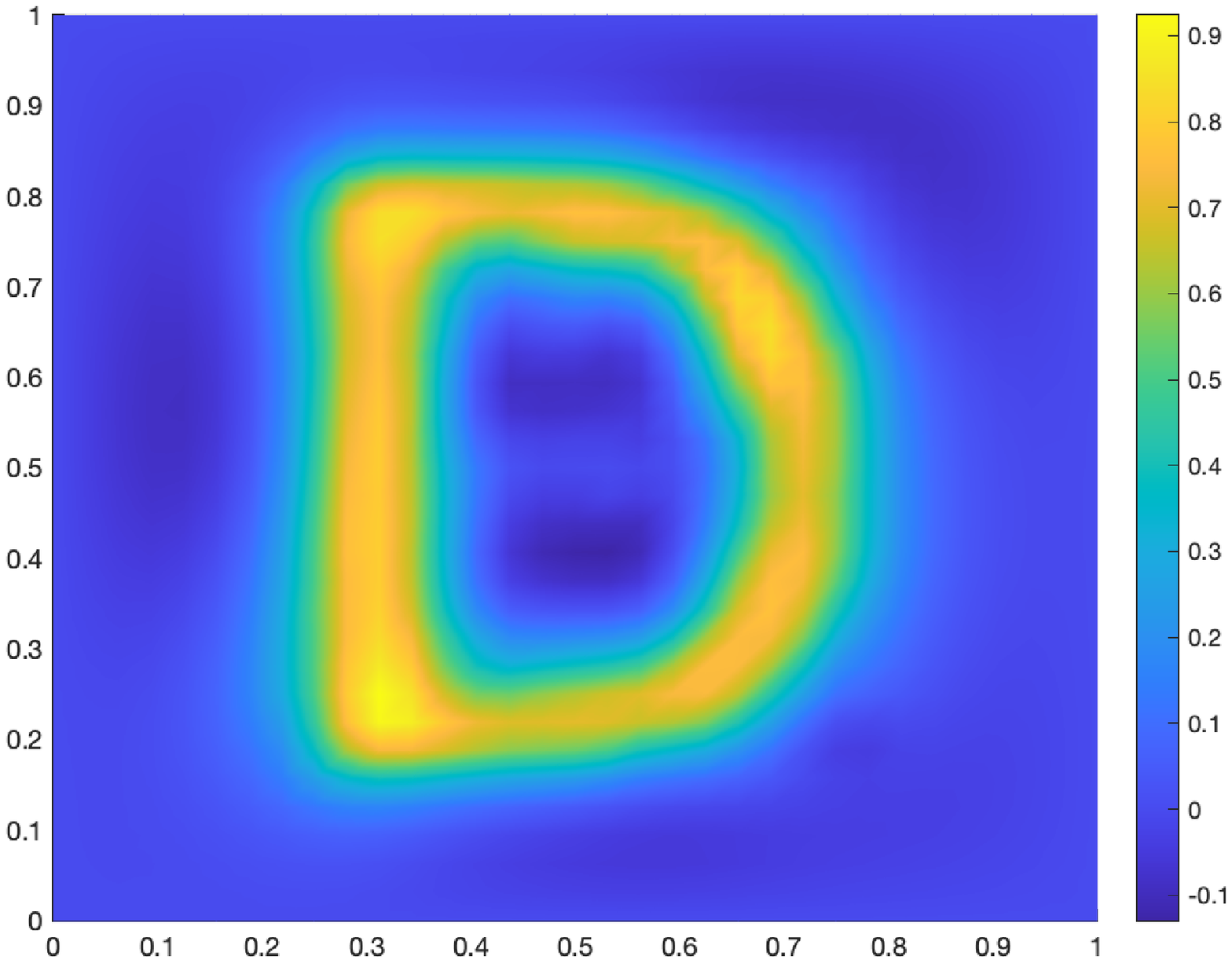}}	\subfigure[O]{\includegraphics[width=0.32\linewidth]{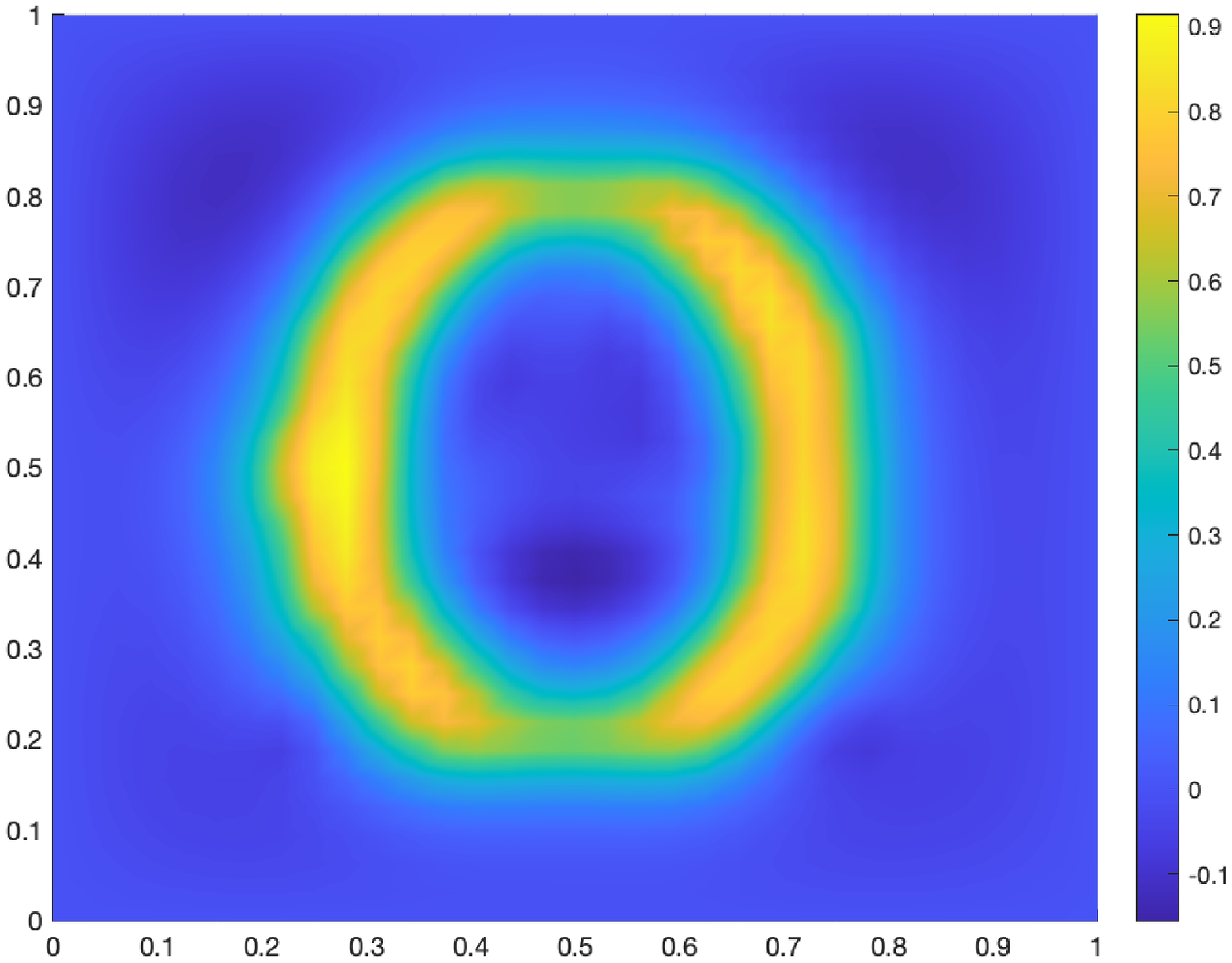}}
	\caption{Recovering letters based on basis constructed from 'A' to 'O'}
	\label{fig:eg1_multi_letters}
\end{figure}
We also compare the POD method with the FEM in the case when $h=1/20$. We construct the POD basis from snapshots of the forward problem with $h=1/20$.  $ 23 $ basis functions are constructed to achieve $10^{-4}$ accuracy. As the DOF during the optimization procedure stays near, the computational cost with POD basis functions is similar between different $h$'s (i.e., $ 47 $ seconds in constructing basis functions and $ 159 $ seconds in solving optimization problems). While with the FEM basis functions, the computation cost drops near quadratically against the DOF. It takes $ 340 $ seconds for $h=1/20$ case, which corresponds to $400$ DOF. In Figure \ref{fig:eg1_A_different_h} we show the comparison between the FEM and POD method with different $h$'s. The observation noise is the same under the same $h$.   Generally speaking, the restoration by the POD method is less affected by the noise than the FEM. This is due to the fact that the noise is perpendicular to the space of the POD basis functions. For a convincing outcome, the POD method is significantly faster than the FEM. 

\begin{figure}[htb]
	\centering
	\subfigure[POD, $h=1/20$]{\includegraphics[width=0.45\linewidth]{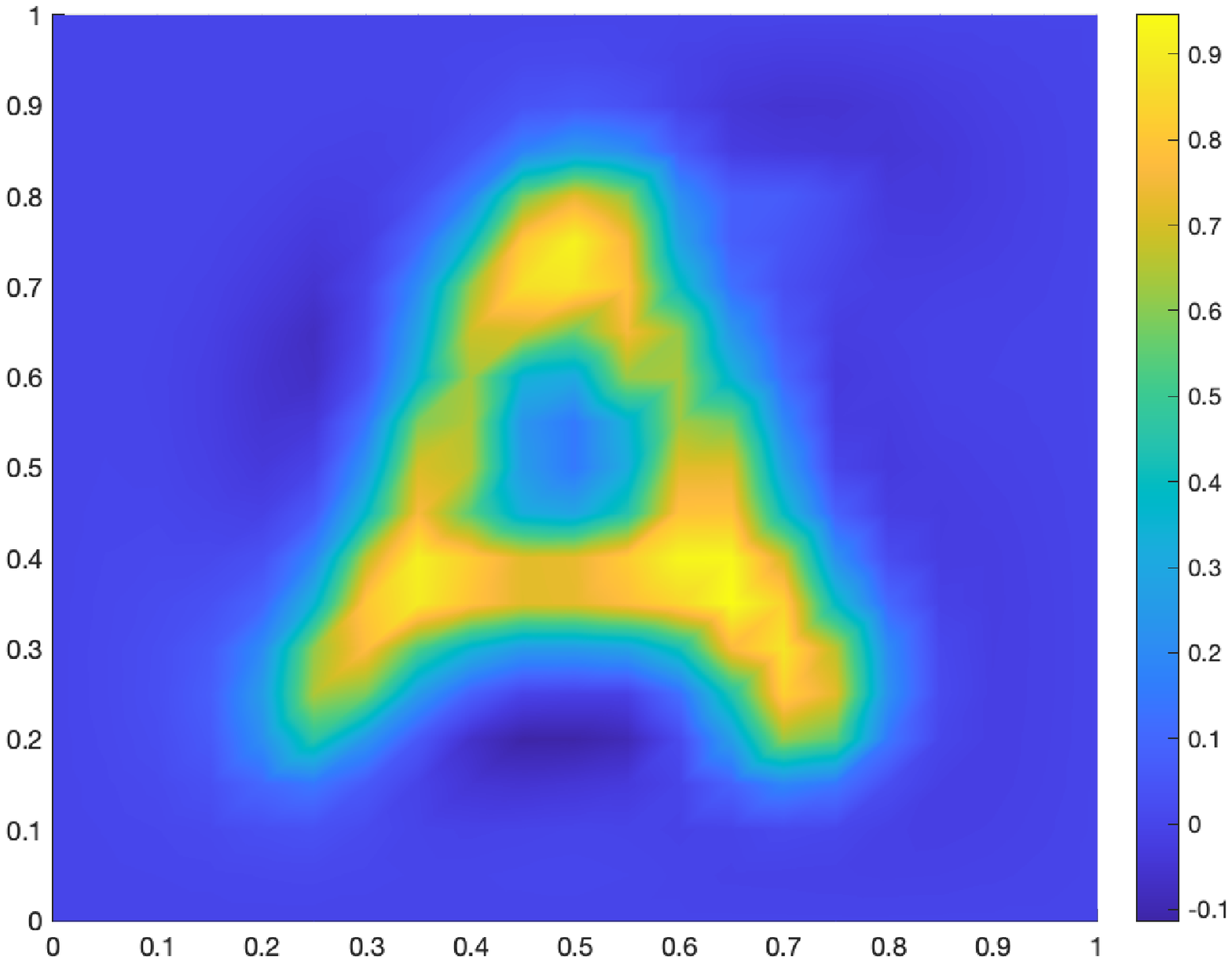}}
	\subfigure[FEM, $h=1/20$]{\includegraphics[width=0.45\linewidth]{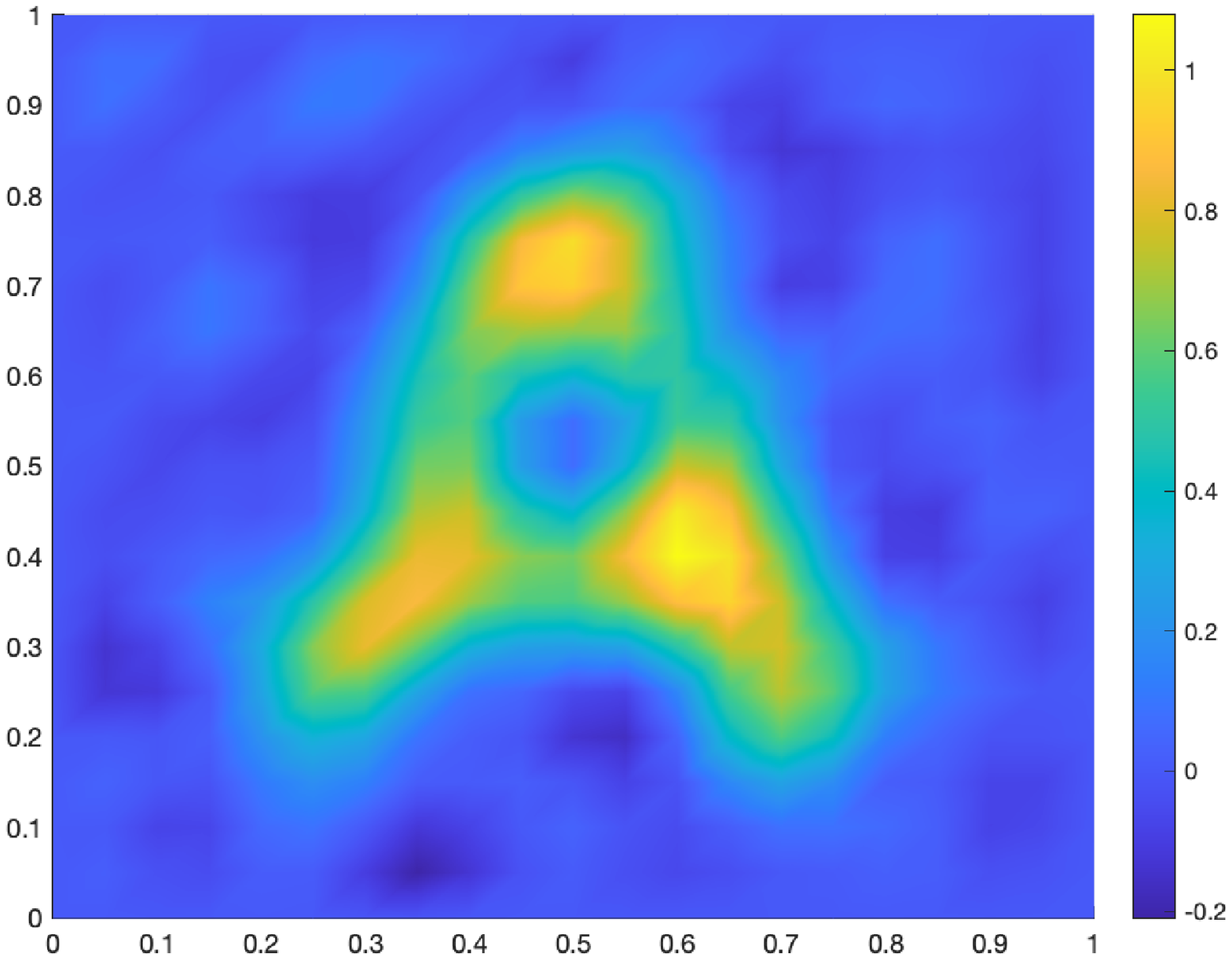}}\\
	\subfigure[POD, $h=1/32$]{\includegraphics[width=0.45\linewidth]{figures_POD/letters_A_n32.eps}}
	\subfigure[FEM, $h=1/32$]{\includegraphics[width=0.45\linewidth]{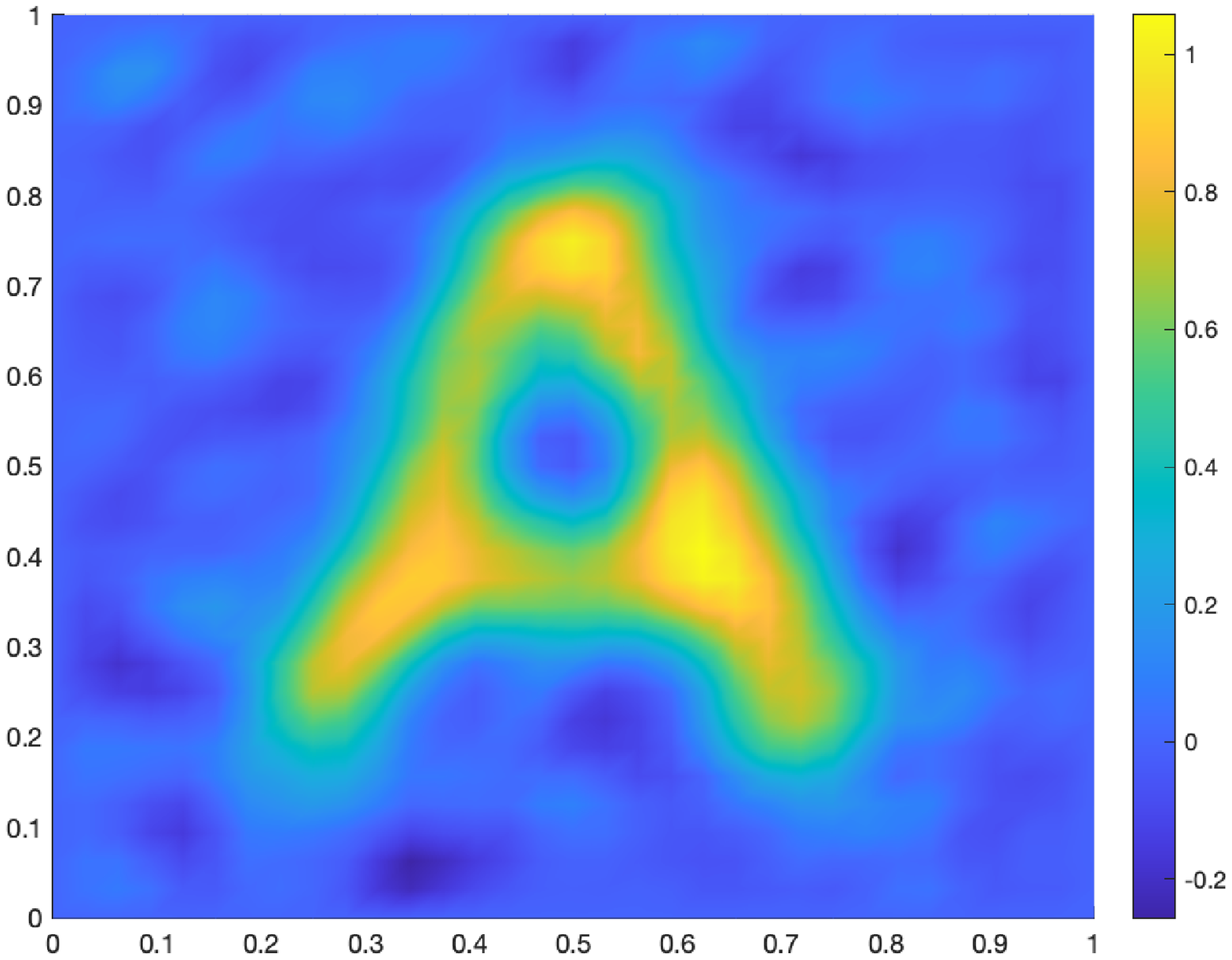}}
	
	\caption{Recovering letter 'A' with different basis and $h$}
	\label{fig:eg1_A_different_h}
\end{figure}
\paragraph{Error Investigation} According to the convergence analysis studied in  Theorem \ref{thm_convergence}, there are three factors that affect the accuracy of our algorithm, i.e., mesh size of FEM basis, regularization factor, and approximation property of the POD basis functions. As such,  we set up three experiments when recovering the letter 'A'. Under each experiment, we generate  $100$ set of observation $y$'s and solve the optimization with the pre-computed POD basis functions. The standard derivation of the noise is still $10^{-3}$. It corresponds to a $10\%$ noise level.  We compute the mean of the error in each set of parameters and show them in Figure  \ref{fig:eg1_err_analysis}. 
\subparagraph{Dependence on the mesh size}   In Figure \ref{fig:eg1_fem_err} we verify that given $\lambda$ when the mesh size $h$ of the FEM basis functions goes down, the error of $Sf$ will first decrease and then converged to a limit that is greater than zero. In the first stage, the order of convergence against mesh size $h$ is $2$. The limit in the second stage depends on $\lambda$ and the dependence is non-monotone. The number of POD basis functions does not change remarkably (also see in the previous example that compares the POD method with FEM) under different mesh sizes. Along with the monotonic decreasing of the error against $h\to 0$, it indicates all sets of POD basis functions all approximate some fixed continuous function space.
\subparagraph{Dependence on the regularization factor} To further investigate the limiting error, in \ref{fig:eg1_lambda_err}, we keep $h=1/32$ fixed and use different regularization factor $\lambda$'s. We find that the optimal $\lambda$ for $h=1/32$ with the $10\%$ noise level is between $10^{-7}$ and $10^{-8}$. The error is a V-shape function near the optimal $\lambda$.   
\subparagraph{Dependence on the number of POD basis functions} The last experiment is to study how the number of POD basis functions affects the error. In \ref{fig:eg1_pod_err}, we show the mean square error with $100$ sets of random generated observation. The error goes down as we increase the number of POD basis functions. The fitted slope in the loglog plot is  $-0.0145$, which indicates the errors exponentially decay with a coefficient $-0.0145$. The p-value for the fitted slope is smaller than $0.001$. 
\begin{figure}[htb]
	\centering
	\subfigure[Mesh size $h$ of the FEM.]{\includegraphics[width=0.32\linewidth]{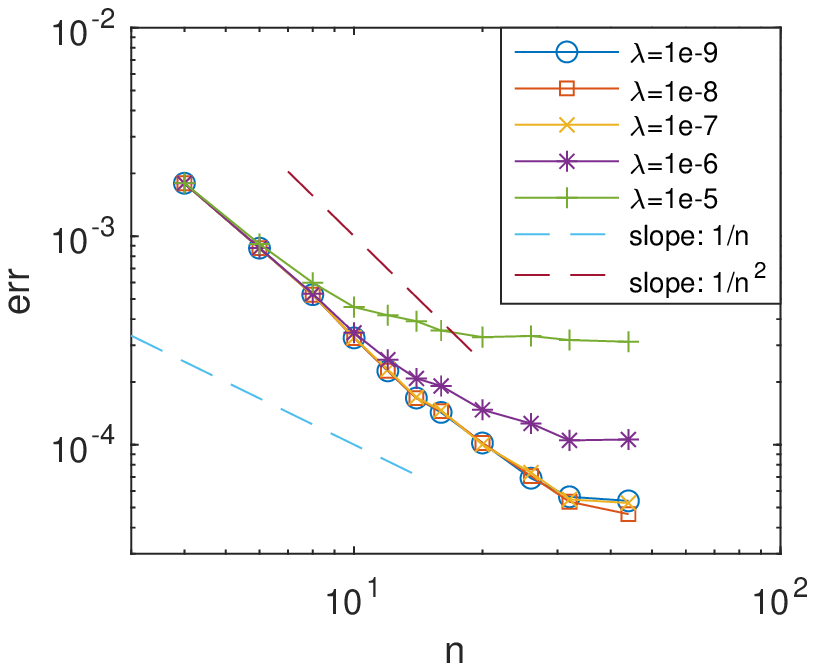}\label{fig:eg1_fem_err}}
	\subfigure[regularization factor  $\lambda$]{\includegraphics[width=0.32\linewidth]{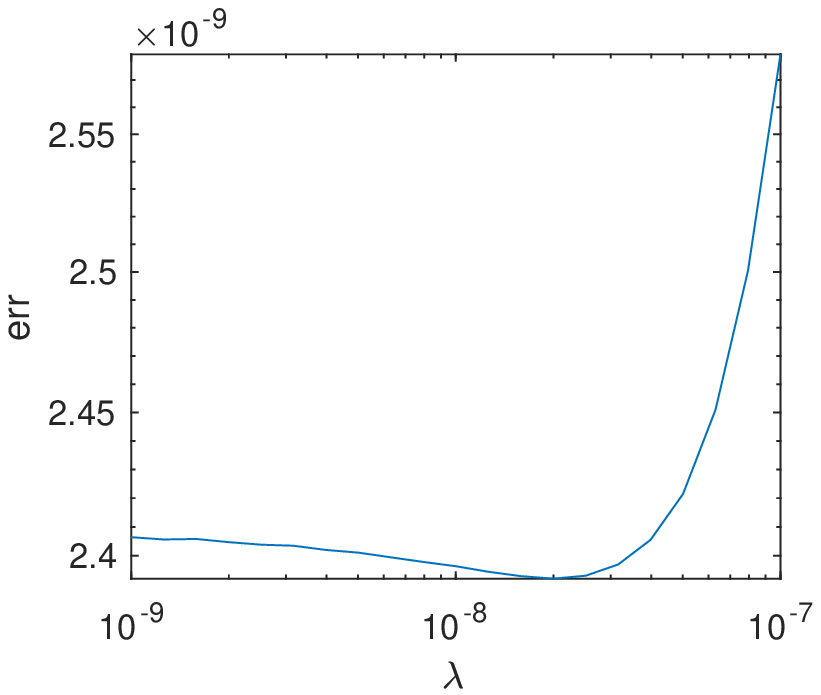}\label{fig:eg1_lambda_err}}
	\subfigure[POD basis number $n$, fitted slope: $1.45\time10^{-2}$]{\includegraphics[width=0.32\linewidth]{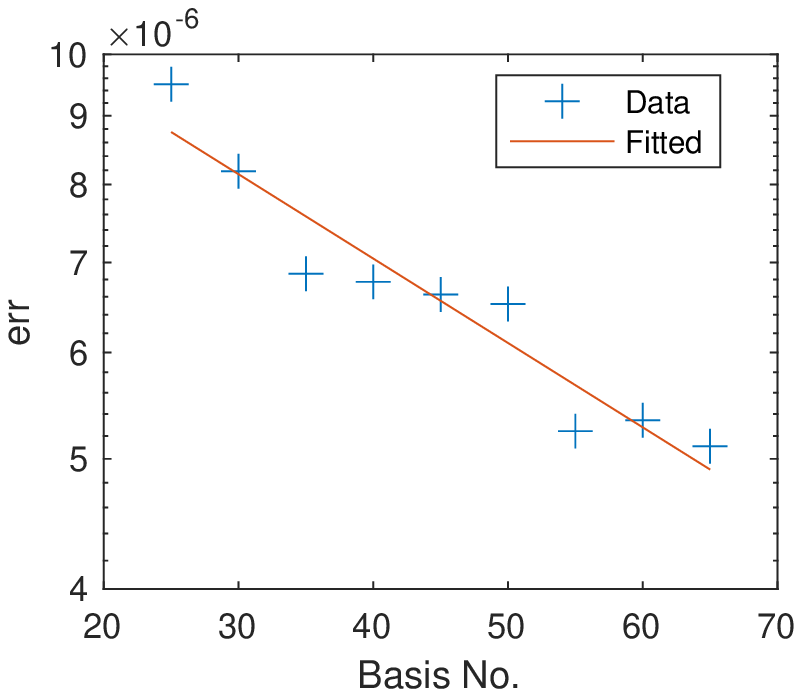}\label{fig:eg1_pod_err}}	
	\caption{Error dependence on different parameters}
	\label{fig:eg1_err_analysis}
\end{figure}
\paragraph{Iterative algorithm for finding the optimal $\lambda$} As analyzed in Eq.\eqref{opti-para}, the optimal $\lambda$ relies on $\|f^*\|^2_{L^2(\Om)}$. While in the setting of inverse problems, the norm of input $f^*$ is unknown. A natural choice is to find the optimal $\lambda$ by using a fixed-point iteration. To be precise, we start from an initial guess like $\lambda_{n, 0}=n^{-4 /(d+4)}$ and solve for $f_0$ with the POD basis functions. After solving $f_j$ the $j$-th step, we update $\lambda$ as,
\begin{align}\label{eqn:opt_lam_alg}
	\lambda_{n, j+1}^{1 / 2+d / 8}=n^{-1 / 2}\left\|S_{\tau, h} f_{j}-m\right\|_{n}\left\|f_{j}\right\|_{L^{2}(\Omega)}^{-1}.
\end{align}
We will stop the iteration of $\lambda$ until $\left\|f_{j}\right\|_{L^{2}(\Omega)}$ converges. As shown in Figure \ref{fig:eg1_opt} such iterative algorithms, with either POD or FEM basis, finds the correct scale of $\lambda$ with a similar pattern. It implies that the POD method captures the whole solution subspace with given known source function and is robust in iteration.
 To be noted, in Eq.\eqref{opti-para}, constant $C$ is not $1$ so the fixed-point iteration Eq.\eqref{eqn:opt_lam_alg} only finds the correct scale of $\lambda$ comparing with Figure \ref{fig:eg1_lambda_err}, in which the error due to $\lambda$ is computed with the knowledge of the ground truth solution $u(\cdot, T)$.
\begin{figure}[h]
	\centering
	\subfigure[by iterating on FEM basis]{\includegraphics[width=0.32\linewidth]{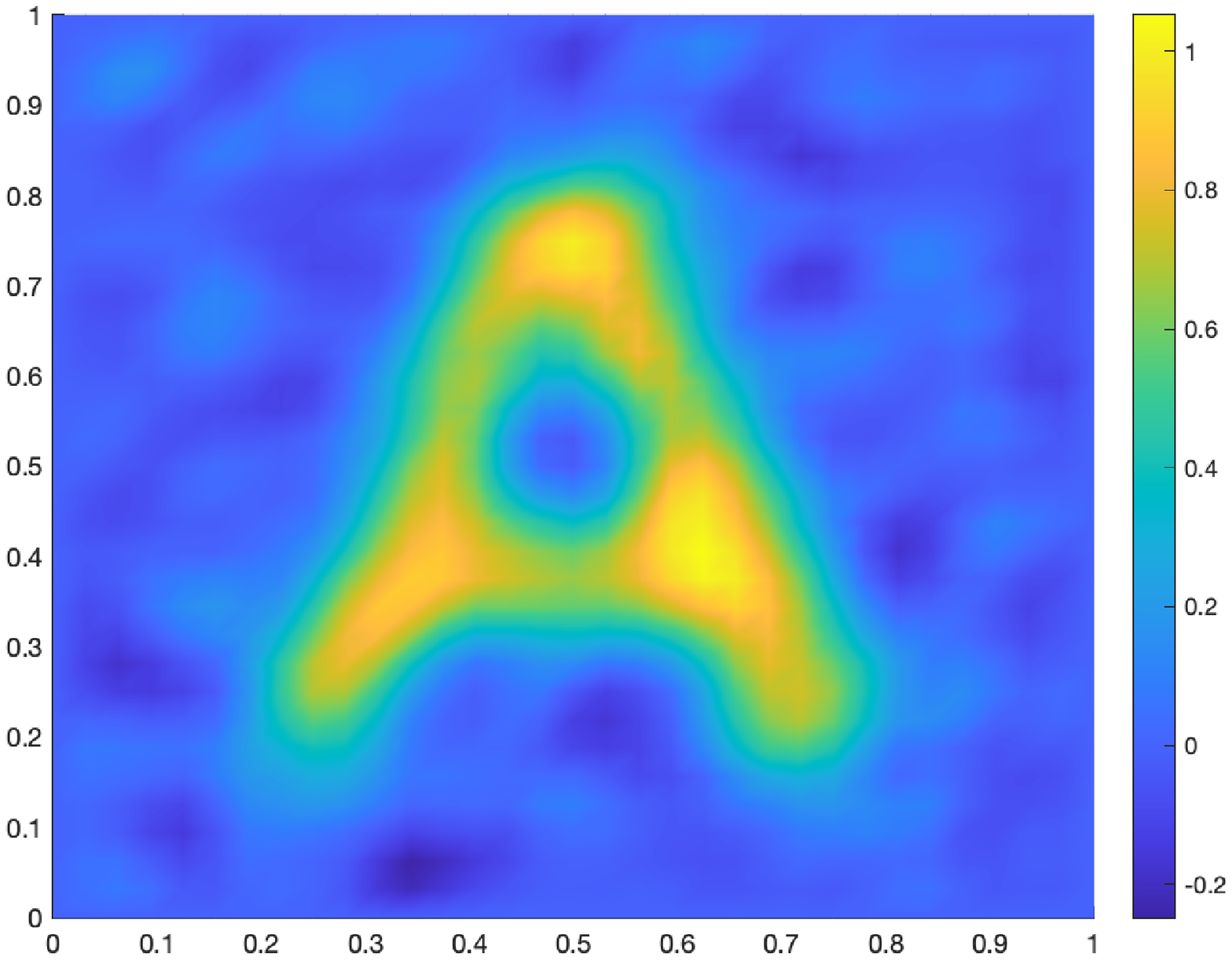}}
	\subfigure[by iteration on POD basis]{\includegraphics[width=0.32\linewidth]{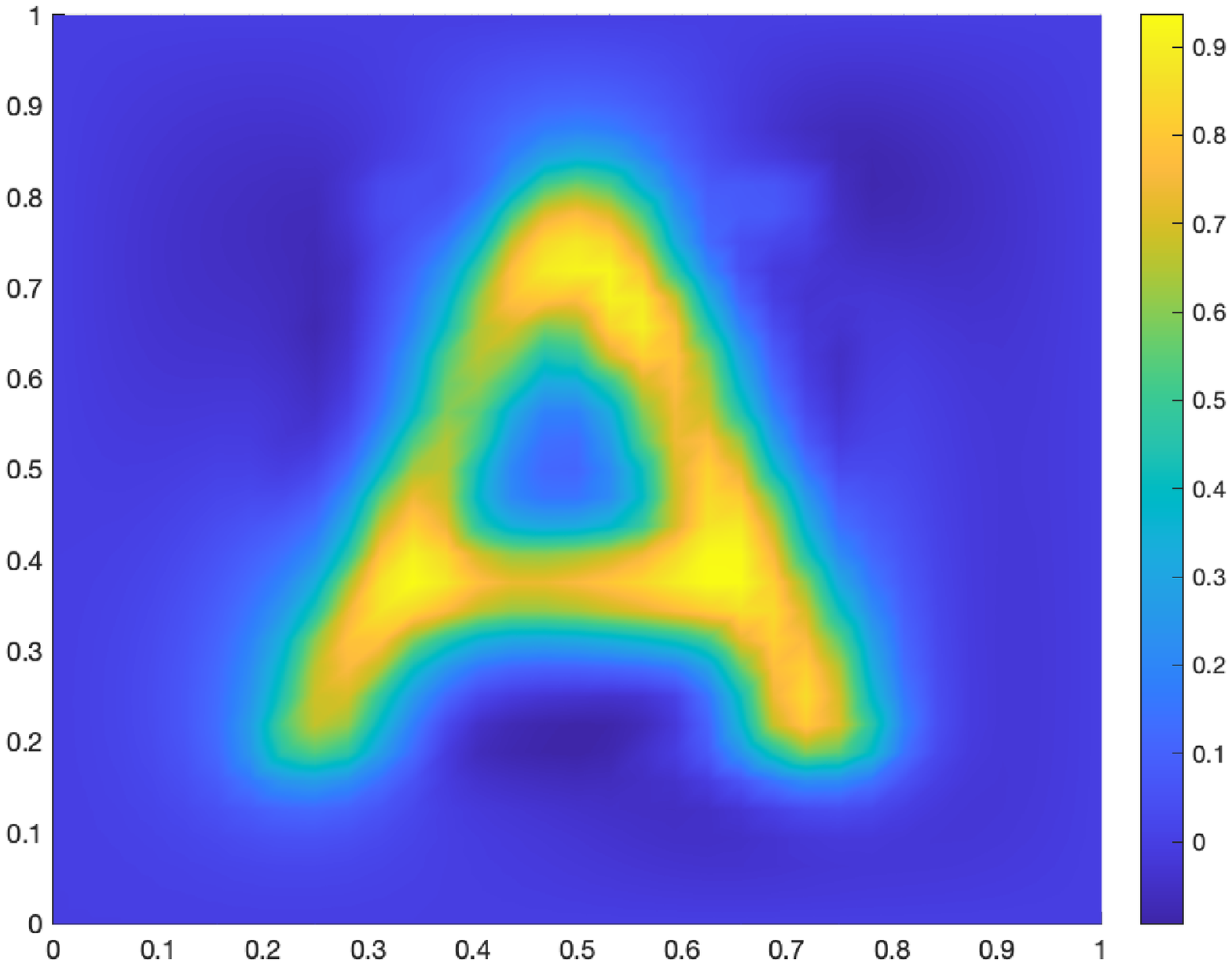}}
	\subfigure[$\lambda_n$ and $L_2$ error after each iteration]{\includegraphics[width=0.32\linewidth]{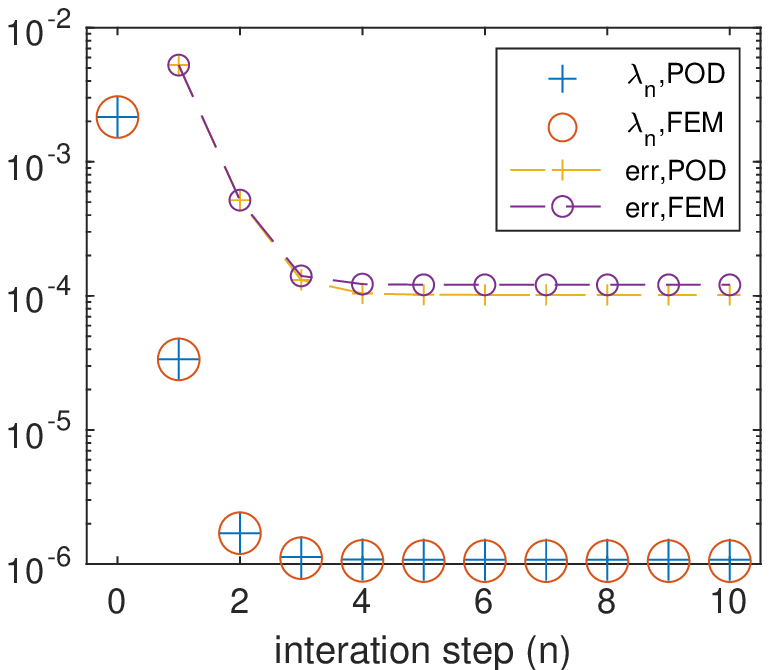}}	
	\caption{Finding optimal $\lambda$ }
	\label{fig:eg1_opt}
\end{figure}

\subsection{Recovering Moving Circles}
In addition to letter examples, we also applied the POD method to recover a ball pattern that moves on some given trajectories in two-dimensional space. Similar to recovering letters,  we construct the POD basis functions from snapshots that are generated by solving forward problems with classic FEM, where the mesh size $h=1/32$ and time step $\Delta t=1/32$. $\lambda$ is set to be $10^{-6}$ and $\sigma=2.75\times 10^{-3}$ which corresponds to a $10\%$ noise level.

\paragraph{Moving along a horizontal line} In such case, the source term $f$ is given by indicator function of set $\Omega_s=\{x\in[0,1]^2|(x_1-s)^2+(x_2-0.5)^2\leq 0.1^2\} $. When constructing the POD basis functions, we use $17$ values of $s$ that are equally distributed between interval $[0.15,0.85]$ including two terminal points. $14$ basis functions are extracted for $10^{-4}$ accuracy calculated by eigenvalues. In \ref{fig:eg2_moving_circles} we use the constructed basis on detecting circles with the same size. We find that the performance of the POD method is still acceptable even the centers of the circle do not coincide with the ones when we construct the POD basis functions. 

\paragraph{Moving along a ring} In addition, as shown in \ref{fig:eg2_moving_circles_2} we  recover the same pattern whose center are on a circle. To be specific, the source term $f$ is given by indicator function of set $\Omega_\theta=\{x\in[0,1]^2|(x_1-0.5-\cos(\theta)/4)^2+(x_2-0.5-\sin(\theta)/4)^2\leq 0.1^2\}$. $33$ basis are extracted for $10^{-4}$ accuracy. Even with a higher DOF, the POD method still yields $10\times$ acceleration comparing with FEM ($60$ seconds vs $615$ seconds).
\begin{figure}[h]
	\centering
	\subfigure[$s=0.3$]{\includegraphics[width=0.19\linewidth]{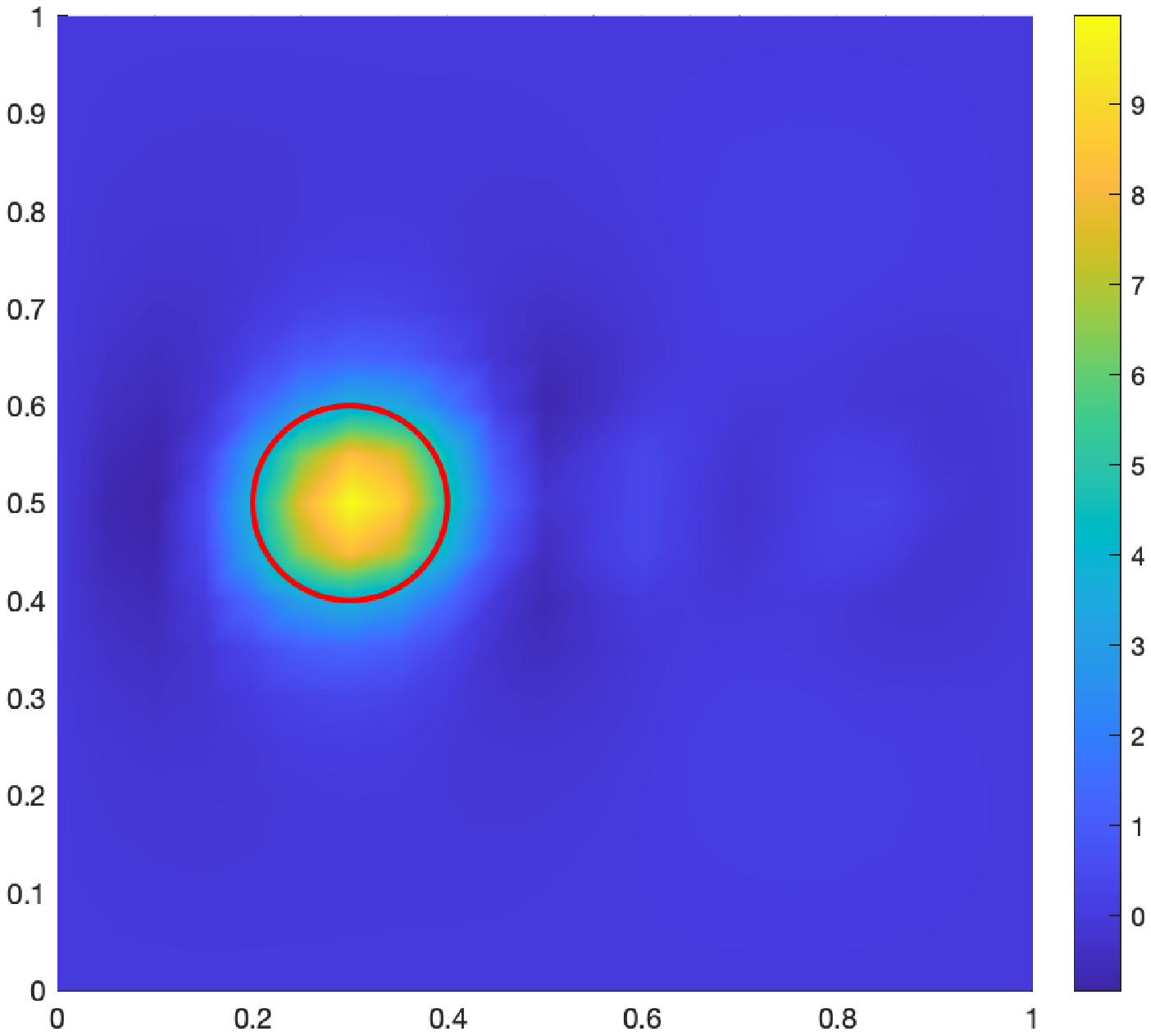}}
	\subfigure[$s=0.4$]{\includegraphics[width=0.19\linewidth]{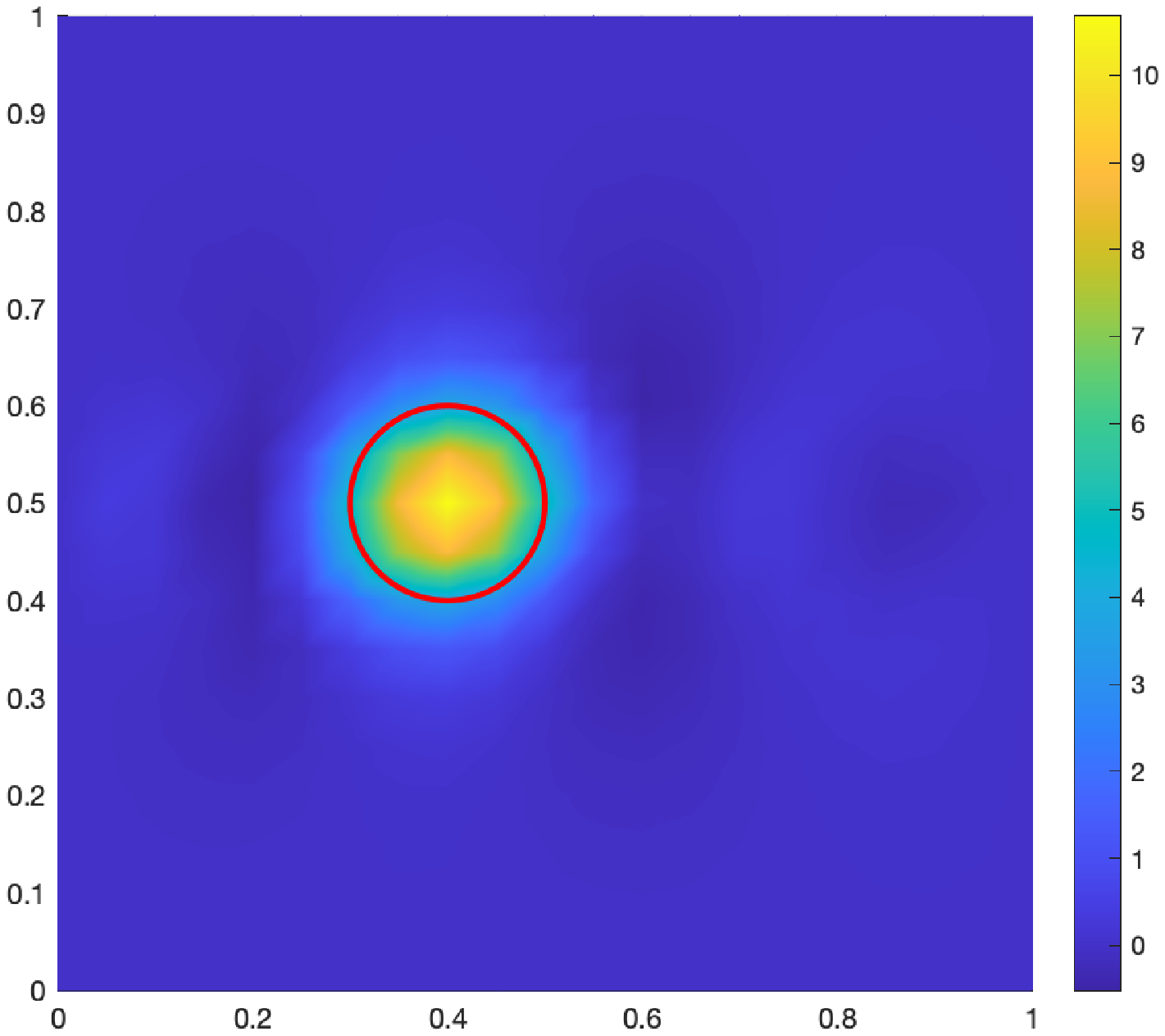}}
	\subfigure[$s=0.5$]{\includegraphics[width=0.19\linewidth]{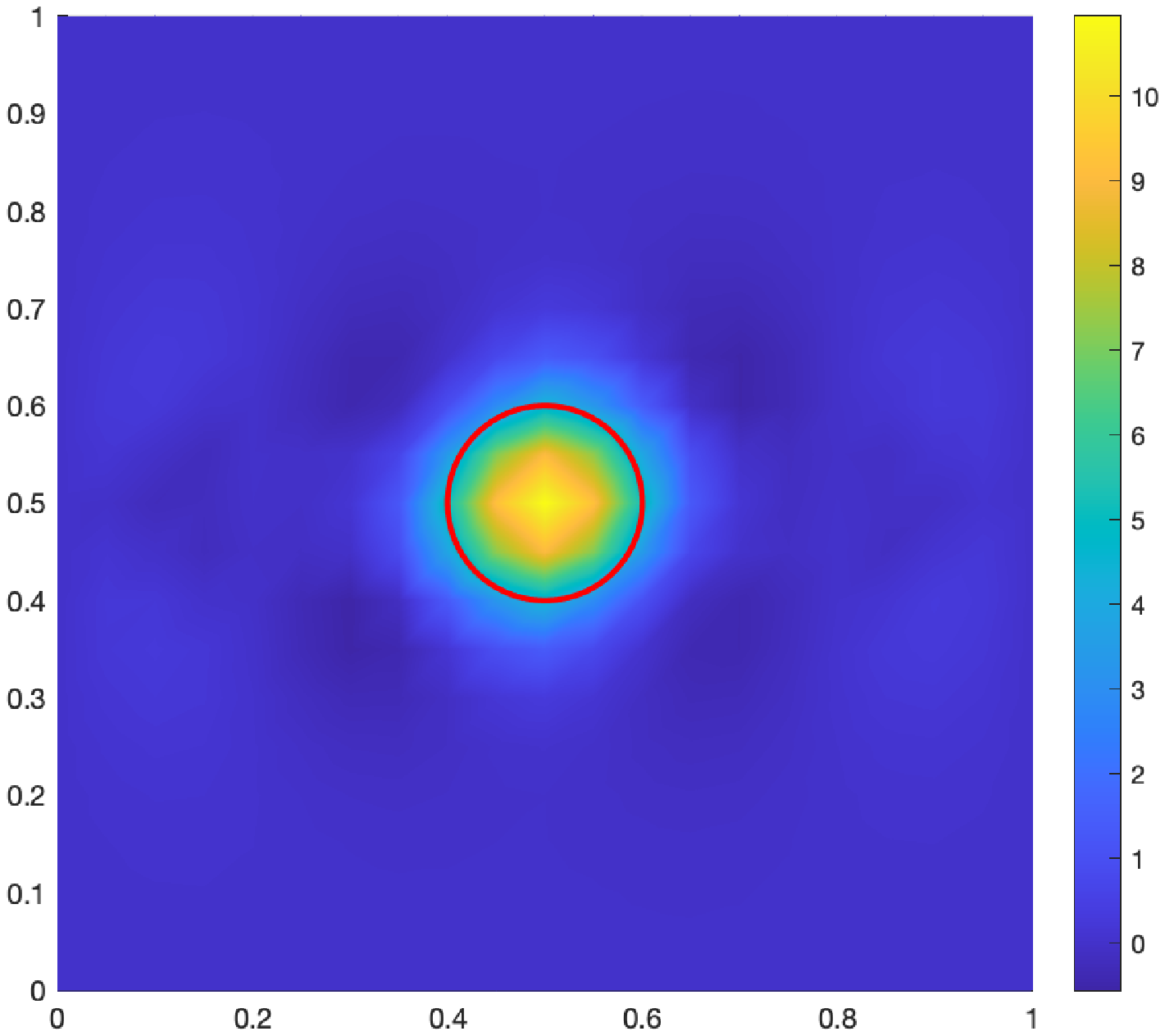}}
	\subfigure[$s=0.6$]{\includegraphics[width=0.19\linewidth]{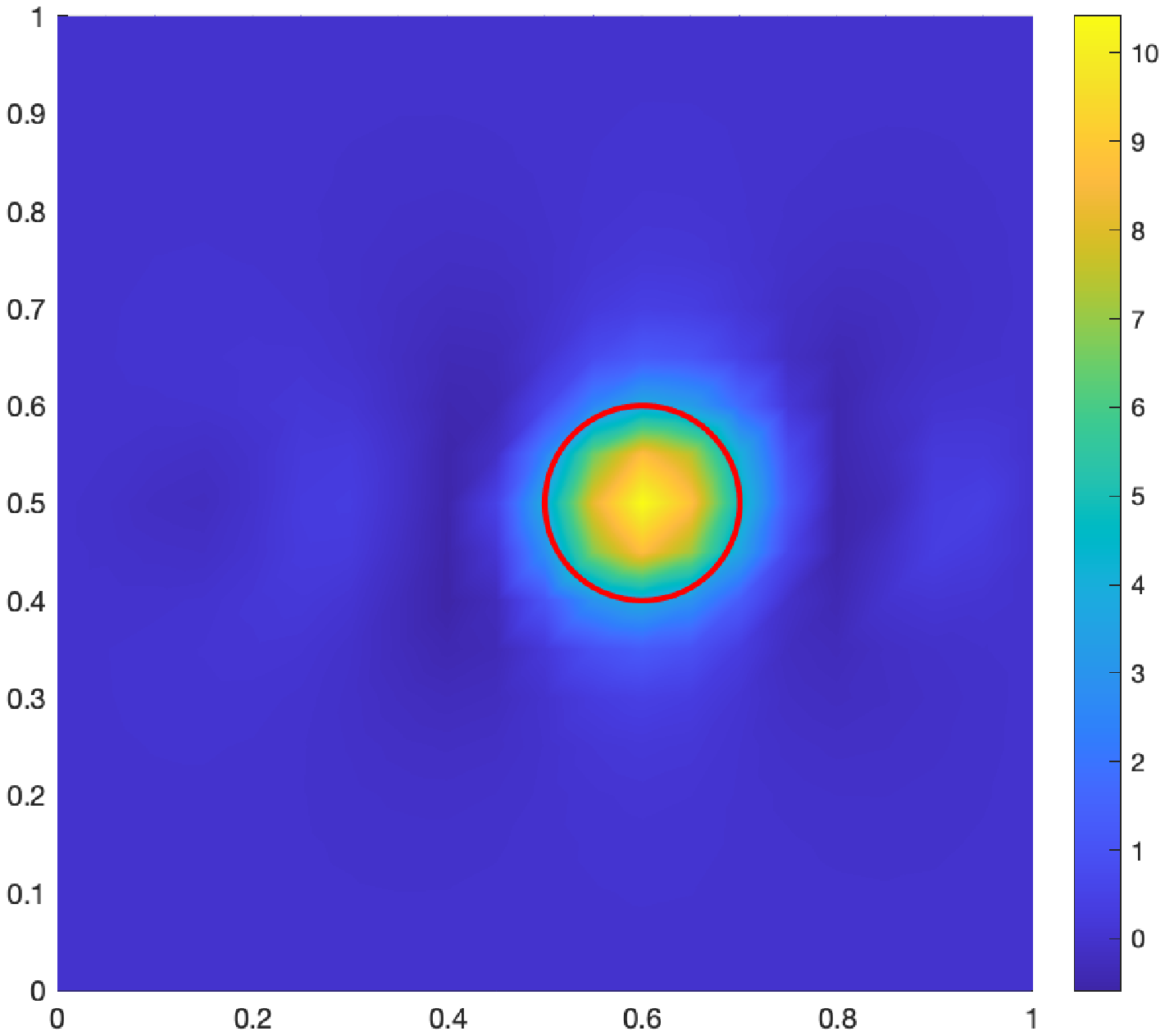}}	
	\subfigure[$s=0.7$]{\includegraphics[width=0.19\linewidth]{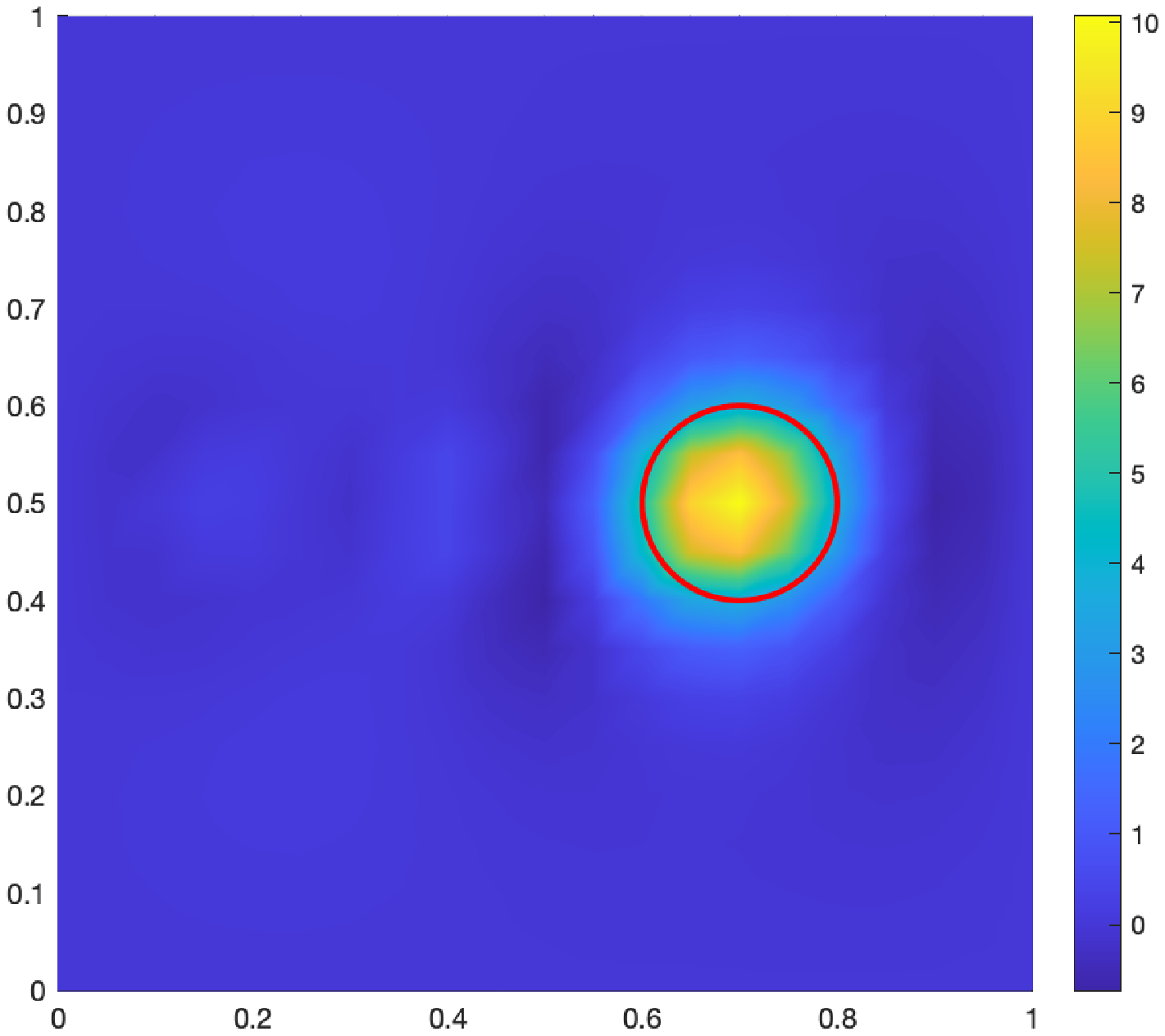}}
	\caption{Recovering Moving Circles with centers point at different $x$, red circle: ground truth}
	\label{fig:eg2_moving_circles}
\end{figure}

\begin{figure}[h]
	\centering
	\subfigure[$\theta=0$]{\includegraphics[width=0.24\linewidth]{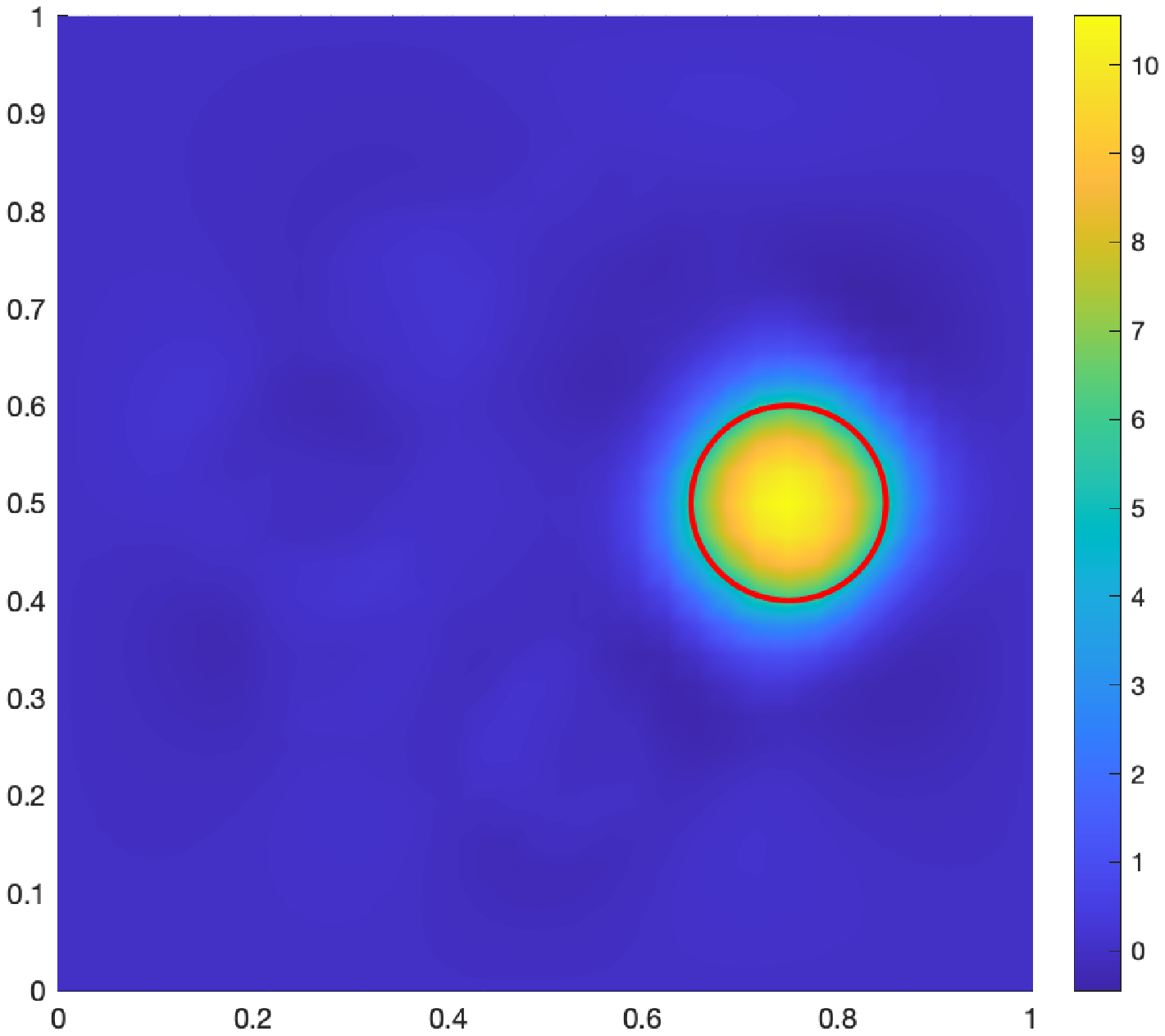}}
	\subfigure[$\theta=\frac{\pi}{4}$]{\includegraphics[width=0.24\linewidth]{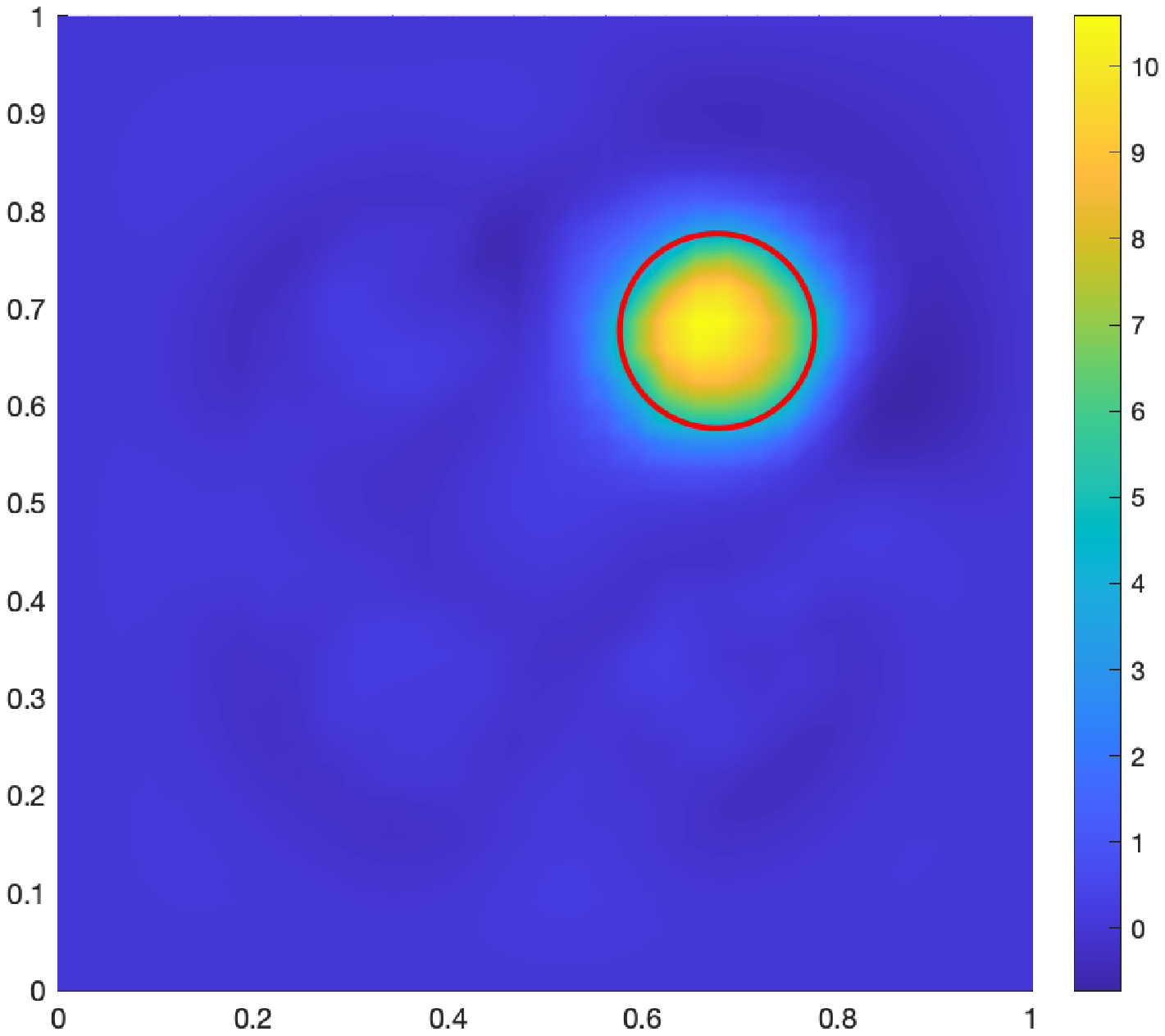}}
	\subfigure[$\theta=\frac{\pi}{2}$]{\includegraphics[width=0.24\linewidth]{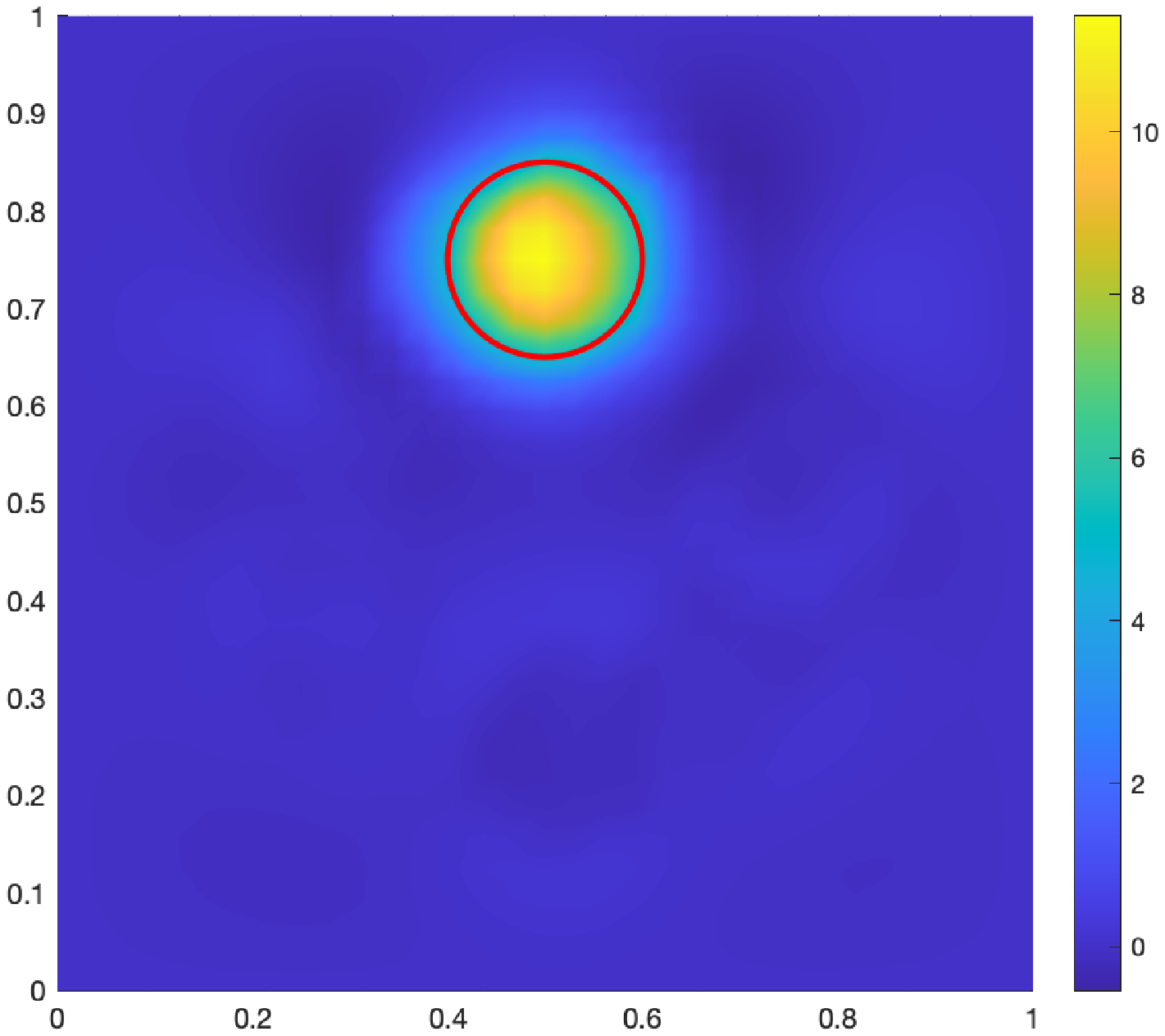}}
	\subfigure[$\theta=\frac{3\pi}{4}$]{\includegraphics[width=0.24\linewidth]{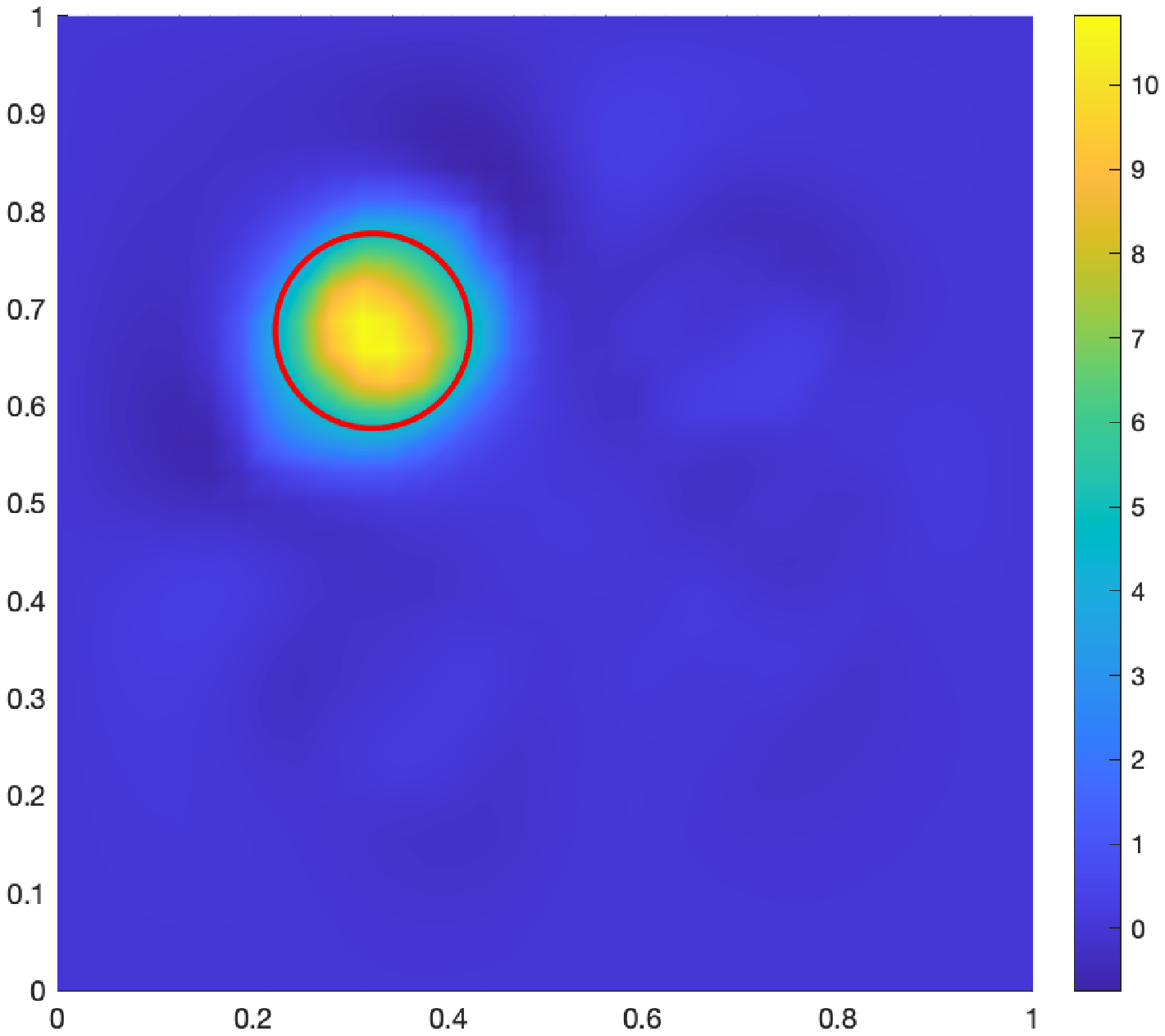}}	\\
	\subfigure[$\theta=\pi$]{\includegraphics[width=0.24\linewidth]{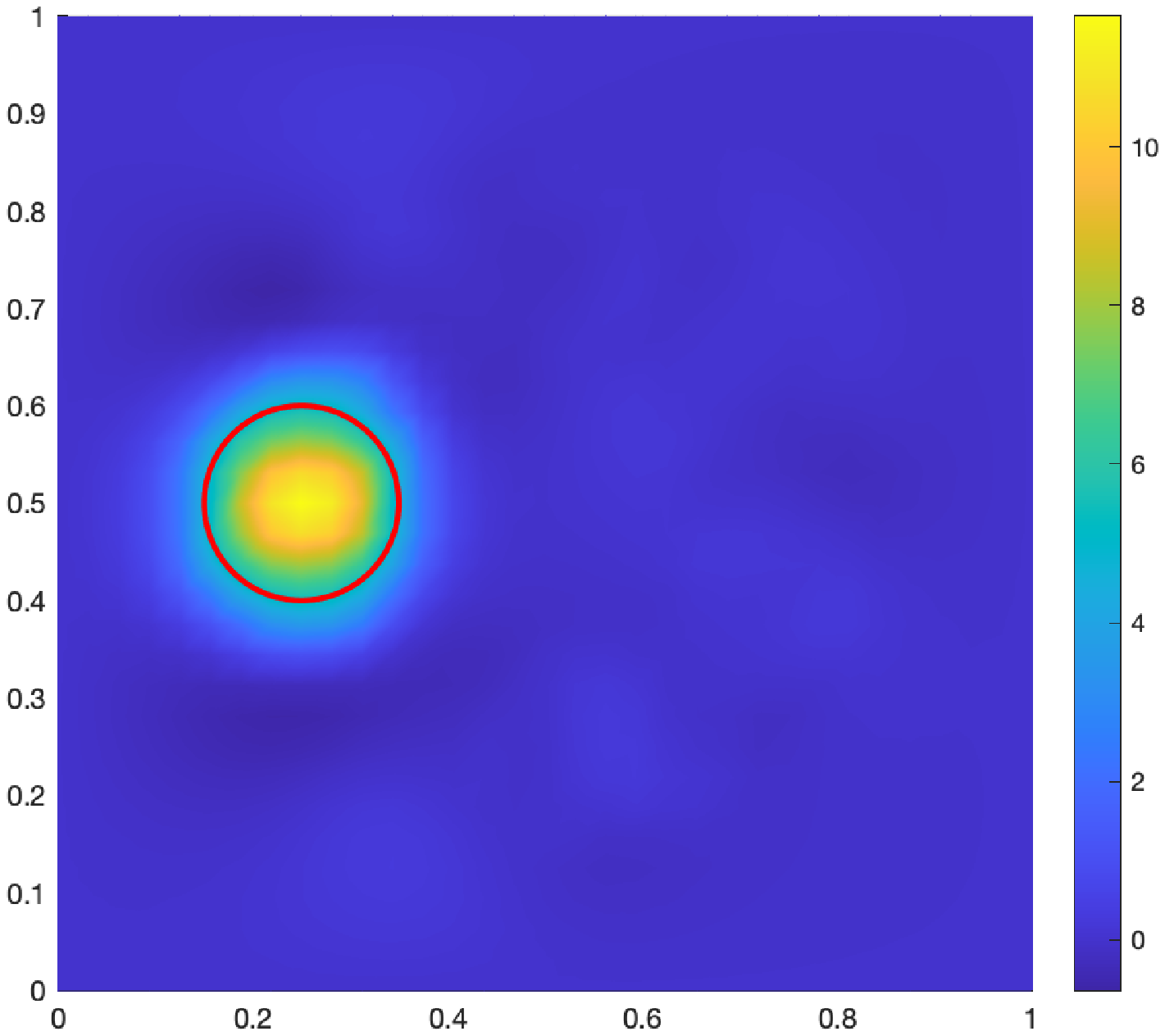}}	
	\subfigure[$\theta=\frac{5\pi}{4}$]{\includegraphics[width=0.24\linewidth]{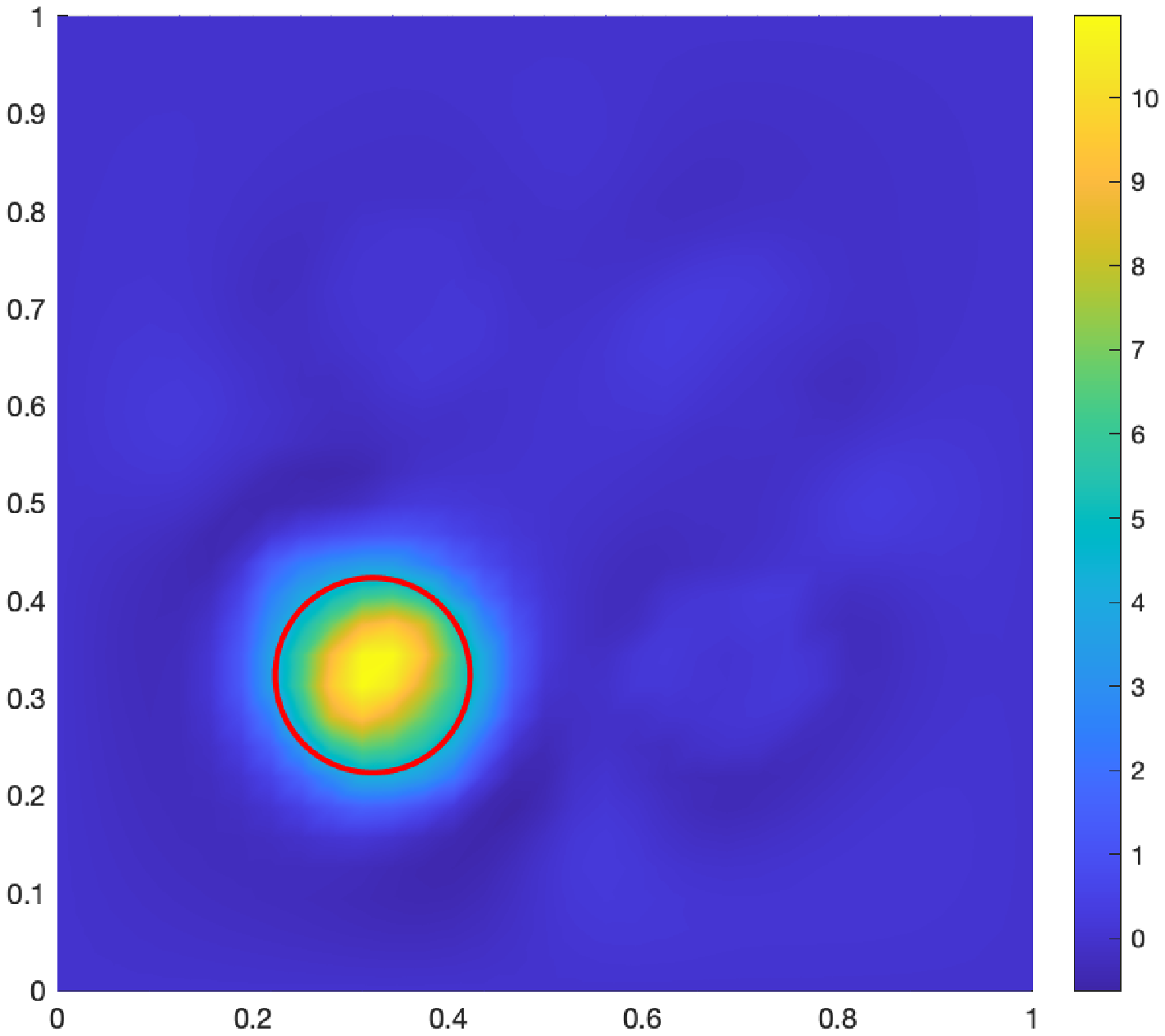}}	
	\subfigure[$\theta=\frac{3\pi}{2}$]{\includegraphics[width=0.24\linewidth]{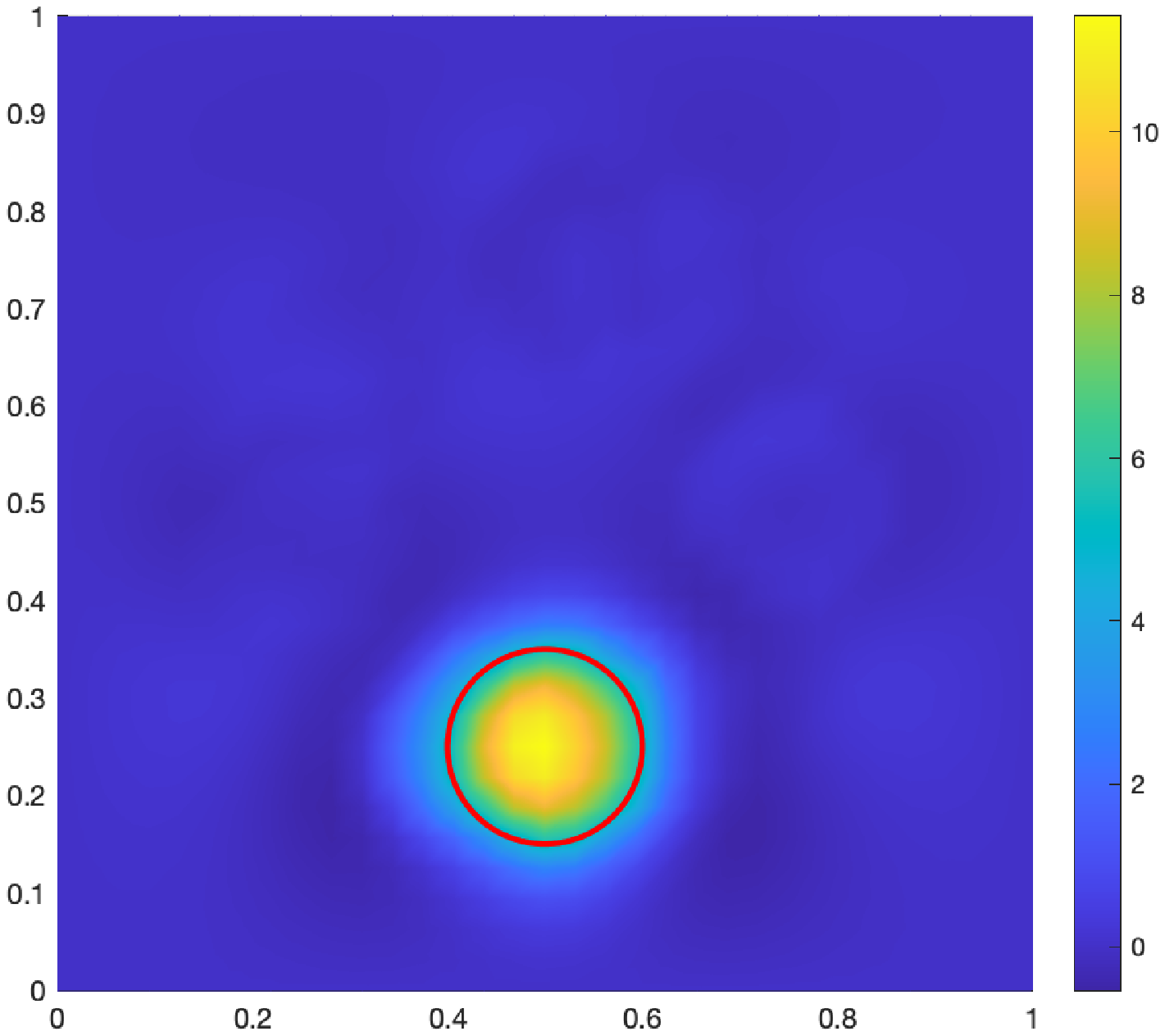}}	
	\subfigure[$\theta=\frac{7\pi}{4}$]{\includegraphics[width=0.24\linewidth]{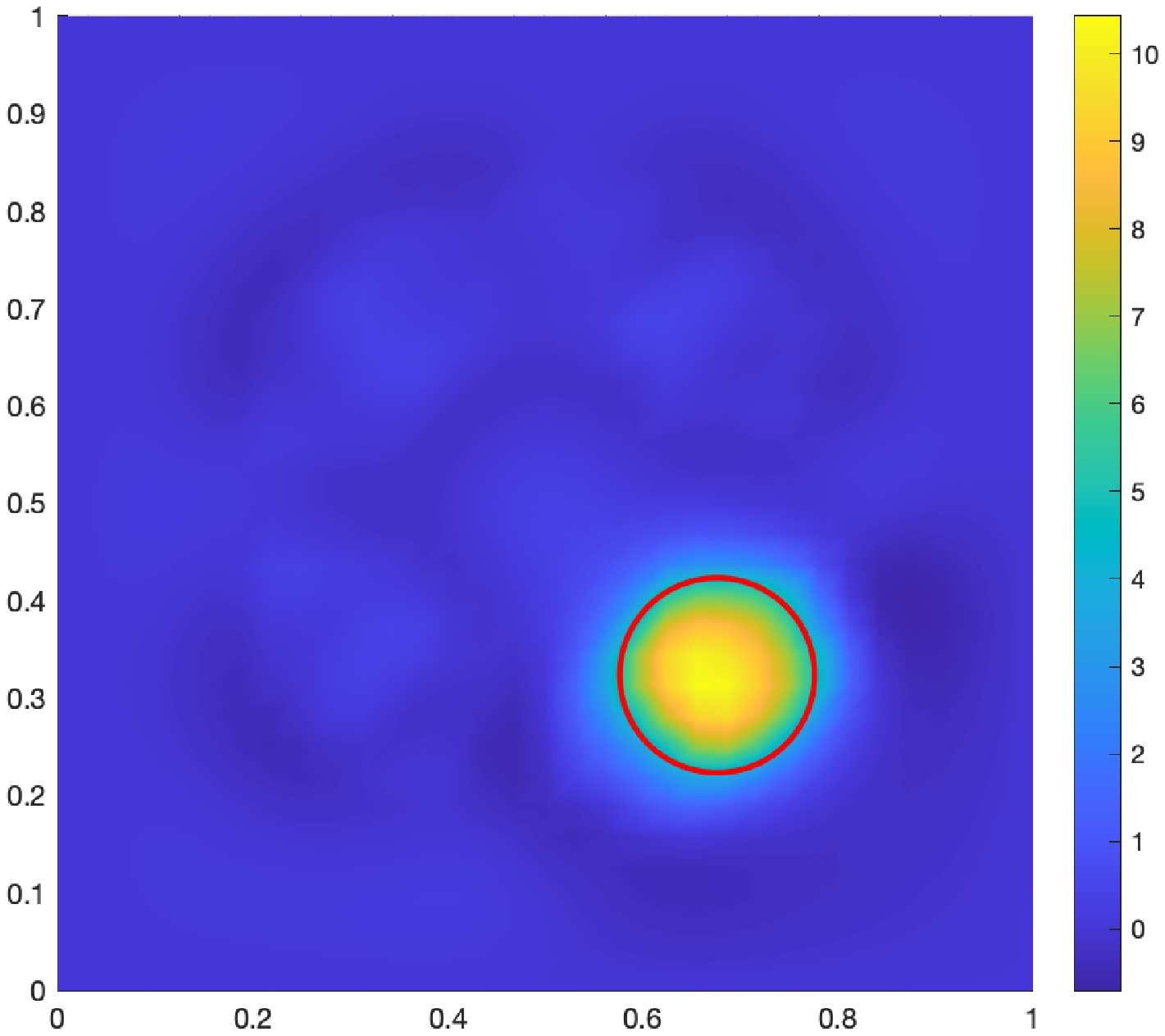}}	
	\caption{Recovering Moving Circles with centers point at different $x$, red circle: ground truth}
	\label{fig:eg2_moving_circles_2}
\end{figure}

\section{Conclusion}\label{Sec:conclusion}
\noindent 
We have developed a data-driven and model-based approach for solving parabolic inverse source problems. The key idea is to exploit (model-based) and construct (data-based) the intrinsic approximate low-dimensional structure of the underlying parabolic PDEs that consists of two components -- a training component that constructs a small number of POD basis functions to achieve significant dimension reduction in the solution space, and a fast algorithm that computes the optimization problem in inverse source problem by using the constructed POD basis functions. Hence, we achieve an effective data and model-based approach for the inverse source problems and overcome the typical computational bottleneck of FEM in solving such problems.

Under a weak assumption on the regularity of the solution, we provide the convergence analysis of our POD algorithm in solving the forward parabolic PDEs and thus obtain the error estimate of the POD algorithm for the parabolic inverse source problems. Finally, we carried out numerical experiments to demonstrate the accuracy and efficiency of the proposed method. Through numerical results, we found that our POD algorithm yields as good approximations as the reference solution obtained by the FEM. But our algorithm can be much cheaper. We expect an even better performance of efficiency can be obtained in 3D problems, which is one of our future topics. We also studied other issues of the POD algorithm, such as the dependence of the error on the mesh size, on the regularization parameter in the least-squares regularized minimization problems, and on the number of POD basis functions.   

\section*{Acknowledgement}
\noindent
The research of W. Zhang is  supported by the National Natural Science Foundation of China No. 11901282 and the Shenzhen Sci-Tech Fund No. RCBS20200714114941241. The research of Z. Zhang is supported by Hong Kong RGC grant projects 17300318 and 17307921, National Natural Science Foundation of China  No. 12171406, Seed Funding Programme for Basic Research (HKU), and Basic Research Programme (JCYJ20180307151603959) of The Science, Technology and Innovation Commission of Shenzhen Municipality.

\bibliographystyle{plain}
\bibliography{ZWpaper}

\begin{thebibliography}{10}

\bibitem{AB05}
V.~Ak{\c{c}}elik, G.~Biros, A.~Draganescu, J.~Hill, O.~Ghattas, and
  B.~Waanders.
\newblock Dynamic data-driven inversion for terascale simulations: real-time
  identification of airborne contaminants.
\newblock In {\em SC'05: Proceedings of the 2005 ACM/IEEE Conference on
  Supercomputing}, pages 43--43. IEEE, 2005.

\bibitem{alla2013time}
A.~Alla and M.~Falcone.
\newblock A time-adaptive {POD} method for optimal control problems.
\newblock {\em IFAC Proceedings Volumes}, 46(26):245--250, 2013.

\bibitem{benner2015survey}
P.~Benner, S.~Gugercin, and K.~Willcox.
\newblock A survey of projection-based model reduction methods for parametric
  dynamical systems.
\newblock {\em SIAM Review}, 57(4):483--531, 2015.

\bibitem{Willcox2015PODsurvey}
P.~Benner, S.~Gugercin, and K.~Willcox.
\newblock A survey of projection-based model reduction methods for parametric
  dynamical systems.
\newblock {\em SIAM Review}, 57(4):483--531, 2015.

\bibitem{berkooz1993POD}
G.~Berkooz, P.~Holmes, and J.~Lumley.
\newblock The proper orthogonal decomposition in the analysis of turbulent
  flows.
\newblock {\em Annual review of fluid mechanics}, 25(1):539--575, 1993.

\bibitem{Chen-Zhang}
Z.~Chen, R.~Tuo, and W.~Zhang.
\newblock Stochastic convergence of a nonconforming finite element method for
  the thin plate spline smoother for observational data.
\newblock {\em SIAM Journal on Numerical Analysis}, 56(2):635--659, 2018.

\bibitem{Chen-Zhang2021}
Z.~Chen, W.~Zhang, and J.~Zou.
\newblock Stochastic convergence of regularized solutions and their finite
  element approximations to inverse source problems.
\newblock {\em arXiv:2104.02352}, 2021.

\bibitem{CHENG2020106213}
Jin Cheng and Jijun Liu.
\newblock An inverse source problem for parabolic equations with local
  measurements.
\newblock {\em Applied Mathematics Letters}, 103:106213, 2020.

\bibitem{CHENG2010142}
Wei Cheng, Ling-Ling Zhao, and Chu-Li Fu.
\newblock Source term identification for an axisymmetric inverse heat
  conduction problem.
\newblock {\em Computers \& Mathematics with Applications}, 59(1):142--148,
  2010.

\bibitem{Johansson2014}
S.~D'haeyer, B.~Johansson, and M.~Slodi{\v{c}}ka.
\newblock Reconstruction of a spacewise-dependent heat source in a
  time-dependent heat diffusion process.
\newblock {\em IMA Journal of Applied Mathematics}, 79(1):33--53, 2014.

\bibitem{new2}
A.~El~Badia, A.~El~Hajj, M.~Jazar, and H.~Moustafa.
\newblock Lipschitz stability estimates for an inverse source problem in an
  elliptic equation from interior measurements.
\newblock {\em Applicable Analysis}, 95(9):1873--1890, 2016.

\bibitem{new3}
A.~El~Badia, T.~Ha~Duong, and F.~Moutazaim.
\newblock Numerical solution for the identification of source terms from
  boundary measurements.
\newblock {\em Inverse Problems in Engineering}, 8(4):345--364, 2000.

\bibitem{Badia2011}
A.~El~Badia and T.~Nara.
\newblock An inverse source problem for {H}elmholtz's equation from the
  {C}auchy data with a single wave number.
\newblock {\em Inverse Problems}, 27(10):105001, 2011.

\bibitem{new4}
G.~Garcia, A.~Osses, and M.~Tapia.
\newblock A heat source reconstruction formula from single internal
  measurements using a family of null controls.
\newblock {\em Journal of Inverse and Ill-posed Problems}, 21(6):755--779,
  2013.

\bibitem{GER83}
S.~Gorelick, B.~Evans, and I.~Remson.
\newblock Identifying sources of groundwater pollution: an optimization
  approach.
\newblock {\em Water Resources Research}, 19(3):779--790, 1983.

\bibitem{gu2021error}
H.~Gu, J.~Xin, and Z.~Zhang.
\newblock Error estimates for a {POD} method for solving viscous {G}-equations
  in incompressible cellular flows.
\newblock {\em SIAM Journal on Scientific Computing}, 43(1):A636--A662, 2021.

\bibitem{hadamard}
J.~Hadamard.
\newblock {\em Lectures on {C}auchy's problem in linear partial differential
  equations}.
\newblock Courier Corporation, 2003.

\bibitem{hesthaven2016certified}
J.~Hesthaven, G.~Rozza, and B.~Stamm.
\newblock {\em Certified reduced basis methods for parametrized partial
  differential equations}.
\newblock Springer, 2016.

\bibitem{holmes:98}
P.~Holmes, J.~Lumley, and G.~Berkooz.
\newblock {\em Turbulence, coherent structures, dynamical systems and
  symmetry}.
\newblock Cambridge University Press, 1998.

\bibitem{new6}
V.~Isakov.
\newblock {\em Inverse problems for partial differential equations}, volume
  127.
\newblock Springer, 2006.

\bibitem{Isakov2013}
V.~Isakov, S.~Leung, and J.~Qian.
\newblock A three-dimensional inverse gravimetry problem for ice with snow
  caps.
\newblock {\em Inverse Problems \& Imaging}, 7(2):523, 2013.

\bibitem{Johansson2007}
T.~Johansson and D.~Lesnic.
\newblock Determination of a spacewise dependent heat source.
\newblock {\em Journal of computational and Applied Mathematics},
  209(1):66--80, 2007.

\bibitem{kunisch2001galerkin}
K.~Kunisch and S.~Volkwein.
\newblock Galerkin proper orthogonal decomposition methods for parabolic
  problems.
\newblock {\em Numerische {M}athematik}, 90(1):117--148, 2001.

\bibitem{kunisch2004hjb}
K.~Kunisch, S.~Volkwein, and L.~Xie.
\newblock {HJB}-{POD}-based feedback design for the optimal control of
  evolution problems.
\newblock {\em SIAM Journal on Applied Dynamical Systems}, 3(4):701--722, 2004.

\bibitem{L08}
X.~Liu.
\newblock Identification of indoor airborne contaminant sources with
  probability-based inverse modeling methods.
\newblock {\em Boulder, CO: PhD dissertation. University of Colorado at
  Boulder}, 2008.

\bibitem{LZ07}
X.~Liu and Z.~Zhai.
\newblock Inverse modeling methods for indoor airborne pollutant tracking:
  literature review and fundamentals.
\newblock {\em Indoor air}, 17(6):419--438, 2007.

\bibitem{nelson20}
P.~Nelson and S.~Yoon.
\newblock Estimation of acoustic source strength by inverse methods: {P}art i,
  conditioning of the inverse problem.
\newblock {\em Journal of sound and vibration}, 233(4):639--664, 2000.

\bibitem{NNR98}
G.~Nunnari, A.~Nucifora, and C.~Randieri.
\newblock The application of neural techniques to the modelling of time-series
  of atmospheric pollution data.
\newblock {\em Ecological Modelling}, 111(2-3):187--205, 1998.

\bibitem{quarteroni2015reduced}
A.~Quarteroni, A.~Manzoni, and F.~Negri.
\newblock {\em Reduced basis methods for partial differential equations: an
  introduction}, volume~92.
\newblock Springer, 2015.

\bibitem{renardy2006introduction}
M.~Renardy and R.~Rogers.
\newblock {\em An introduction to partial differential equations}, volume~13.
\newblock Springer Science \& Business Media, 2006.

\bibitem{sirovich1987}
L.~Sirovich.
\newblock Turbulence and the dynamics of coherent structures. {I}. {C}oherent
  structures.
\newblock {\em Quarterly of applied mathematics}, 45(3):561--571, 1987.

\bibitem{Sirovich:1987}
L.~Sirovich.
\newblock Turbulence and the dynamics of coherent structures. {I}. {C}oherent
  structures.
\newblock {\em Quarterly of applied mathematics}, 45(3):561--571, 1987.

\bibitem{thomee1990finite}
V.~Thom{\'e}e.
\newblock The finite element method for parabolic problems.
\newblock In {\em Mathematical Theory of Finite and Boundary Element Methods},
  pages 135--218. Springer, 1990.

\bibitem{volkwein2013proper}
S.~Volkwein.
\newblock Proper orthogonal decomposition: {T}heory and reduced-order
  modelling.
\newblock {\em Lecture Notes, University of Konstanz}, 4(4), 2013.

\bibitem{WY10}
J.~Wong and P.~Yuan.
\newblock A {FE}-based algorithm for the inverse natural convection problem.
\newblock {\em International Journal for numerical methods in fluids},
  68(1):48--82, 2012.

\end{thebibliography}

\end{document}